\documentclass[10pt]{article}
\usepackage{stackengine}
\usepackage{etex}
\usepackage[all]{xy}
\usepackage{cite}
\usepackage{amsfonts}
\usepackage{amsthm}
\usepackage{enumerate}
\usepackage{graphicx}
\usepackage{mathrsfs}
\usepackage{bm}
\usepackage{amssymb,amsmath} % font used for R in Real numbers
\usepackage{makecell}
\usepackage{tikz}
\usepackage{pgf}
\usepackage{tikz}
\usetikzlibrary{patterns}
\usepackage{pgffor}
\usepackage{pgfcalendar}
\usepackage{pgfpages}
\usepackage{shuffle,yfonts}
\usepackage{mathtools}
\DeclareFontFamily{U}{shuffle}{}
\DeclareFontShape{U}{shuffle}{m}{n}{ <-8>shuffle7 <8->shuffle10}{}

\def\cbar#1{\wc{\bar#1}}

\newcommand{\AMZV}{\mathsf {AMZV}}

\newcommand{\MZV}{\mathsf {MZV}}

\newcommand{\MTV}{\mathsf {MTV}}
\newcommand{\MSV}{\mathsf {MSV}}
\newcommand{\MMV}{\mathsf {MMV}}
\newcommand{\MMVo}{\mathsf {MMVo}}
\newcommand{\MMVe}{\mathsf {MMVe}}
\newcommand{\AMMVe}{\mathsf {AMMVe}}
\newcommand{\AMMVo}{\mathsf {AMMVo}}
\newcommand{\AMMV}{\mathsf {AMMV}}
\newcommand{\AMTV}{\mathsf {AMTV}}
\newcommand{\AMtV}{\mathsf {AMtV}}
\newcommand{\AMSV}{\mathsf {AMSV}}
\newcommand{\CMZV}{\mathsf {CMZV}}
\newcommand{\sha}{\shuffle}
\newcommand{\cst}{\mathbin{\rotatebox[origin=c]{180}{$\shuffle$}}}
\newcommand{\cstt}{\mathbin{\rotatebox[origin=c]{180}{$\scriptstyle \sha$}}}

\newcommand\brabar{\scalebox{.3}{(\,}\raisebox{-2.1pt}{--}\scalebox{.3}{\,)}}

\newcommand{\db}{{\mathbb D}}
\newcommand{\pari}{{\rm par}}
\newcommand{\dk}{{\mathbb K}}
\newcommand{\ga}{\alpha}
\newcommand{\gb}{\beta}
\newcommand{\gk}{\kappa}
\newcommand{\gl}{\lambda}
\newcommand{\gd}{\delta}
\newcommand{\gs}{\sigma}

\newcommand{\lra}{\longrightarrow}
\newcommand{\Lra}{\Longrightarrow}

\newcommand\setX{{\mathsf{X}}}
\newcommand\fA{{\mathfrak{A}}}
\newcommand\evaM{{\texttt{M}}}
\newcommand\evaML{{\text{\em{\texttt{M}}}}}
\newcommand\tz{{\texttt{z}}}
\newcommand\emz{\emph{\texttt{z}}}
\newcommand\tx{{\texttt{x}}}
\newcommand\om{{\omega}}

\newcommand\eps{{\varepsilon}}
\newcommand{\bfMB}{{\bf MB}}
\newcommand{\bftB}{{\bf tB}}
\newcommand{\bfTB}{{\bf TB}}
\newcommand{\bfSB}{{\bf SB}}

\newcommand{\bfp}{{\bf p}}
\newcommand{\bfq}{{\bf q}}

\newcommand{\bfu}{{\bf u}}
\newcommand{\bfv}{{\bf v}}
\newcommand{\bfw}{{\bf w}}

\newcommand{\bfk}{{\boldsymbol{\sl{k}}}}
\newcommand{\bfl}{{\boldsymbol{\sl{l}}}}

\newcommand{\bfn}{{\boldsymbol{\sl{n}}}}
\newcommand{\bfs}{{\boldsymbol{\sl{s}}}}
\newcommand{\bft}{{\boldsymbol{\sl{t}}}}

\newcommand{\bfz}{{\boldsymbol{\sl{z}}}}
\newcommand\bfgs{{\boldsymbol \gs}}

\newcommand\bfxi{{\boldsymbol \xi}}

\newcommand\bfeps{{\boldsymbol \eps}}
\newcommand\bfone{{\bf 1}}
\newcommand{\myone}{{1}}

\usetikzlibrary{arrows,shapes,chains}
 \allowdisplaybreaks

\catcode`!=11
\let\!int\int \def\int{\displaystyle\!int}
\let\!lim\lim \def\lim{\displaystyle\!lim}
\let\!sum\sum \def\sum{\displaystyle\!sum}
\let\!sup\sup \def\sup{\displaystyle\!sup}
\let\!inf\inf \def\inf{\displaystyle\!inf}
\let\!cap\cap \def\cap{\displaystyle\!cap}
\let\!max\max \def\max{\displaystyle\!max}
\let\!min\min \def\min{\displaystyle\!min}
\let\!frac\frac \def\frac{\displaystyle\!frac}
\catcode`!=12

\let\oldsection\section
\renewcommand\section{\setcounter{equation}{0}\oldsection}

\allowdisplaybreaks

\DeclareMathOperator*{\sgn}{sgn}
\DeclareMathOperator*{\dep}{dep}
\DeclareMathOperator*{\RRe}{Re}
\DeclareMathOperator{\Li}{Li}

\newcommand{\od}{\mathrlap{\mathrm{o}}\phantom{+}}
\newcommand{\ev}{\mathrlap{\mathrm{e}}\phantom{+}}

\DeclareFontFamily{U}{mathx}{}
\DeclareFontShape{U}{mathx}{m}{n}{<-> mathx10}{}
\DeclareSymbolFont{mathx}{U}{mathx}{m}{n}
\DeclareMathAccent{\widehat}{0}{mathx}{"70}
\DeclareMathAccent{\wc}{0}{mathx}{"71}

\def\R{\mathbb{R}}

\def\N{\mathbb{N}}

\def\Q{\mathbb{Q}}
\def\CC{\mathbb{C}}

\def\t{\widetilde{t}}

\def\ze{\zeta}

\def\ol{\overline}

\theoremstyle{plain}
\newtheorem{thm}{Theorem}[section]
\newtheorem{lem}[thm]{Lemma}
\newtheorem{cor}[thm]{Corollary}
\newtheorem{con}[thm]{Conjecture}
\newtheorem{pro}[thm]{Proposition}
\theoremstyle{definition}
\newtheorem{defn}[thm]{Definition}
\newtheorem{re}[thm]{Remark}
\newtheorem{exa}[thm]{Example}

\begin{document}
%%%%%%%%%%%%%%%%%%%% title %%%%%%%%%%%%%%%%%%%%%%%%%%%%%%%%%%%%%%%%%%%%%%%%
\title{\bf Alternating Multiple Mixed Values: Regularization, Special Values, Parity and Dimension Conjectures}
\author{{Ce Xu${}^{a,}$\thanks{Email: cexu2020@ahnu.edu.cn},\ \  {Lu Yan${}^{b,}$\thanks{Email: 1910737@tongji.edu.cn}}{}\ \ and \ Jianqiang Zhao${}^{c,}$\thanks{Email: zhaoj@ihes.fr}}\\[1mm]
\small a. School of Mathematics and Statistics, Anhui Normal University, Wuhu 241002, P.R. China\\
\small b. School of Mathematical Sciences, Tongji University, Shanghai 200092, P.R. China\\
\small c. Department of Mathematics, The Bishop's School, La Jolla, CA 92037, United States of America}

\date{}
\maketitle

\noindent{\bf Abstract.} In this paper, we define and study a variant of multiple zeta values (MZVs) of level four, called alternating multiple mixed values or alternating multiple $M$-values (AMMVs), forming a $\Q[i]$-subspace of the colored MZVs of level four. This variant includes the alternating version of Hoffman's multiple $t$-values, Kaneko-Tsumura's multiple $T$-values, and the multiple $S$-values studied by the authors previously as special cases. We exhibit nice properties similar to the ordinary MZVs such as the generalized duality, integral shuffle and series stuffle relations. After setting up the algebraic framework we derive the regularized double shuffle relations of the AMMVs by adopting the machinery from color MZVs of level four. As an important application, we prove a parity result for AMMVs previously conjectured by us.  We also investigate several alternating multiple $S$- and $T$-values by establishing some explicit relations of integrals involving arctangent function. At the end, we compute the dimensions of a few interesting subspaces of AMMVs for weight less than 9. Supported by theoretical and numerical evidence aided by numerical and symbolic computation, we formulate a few conjectures concerning the dimensions of the above-mentioned subspaces of AMMVs. These conjectures hint at a few very rich but previously overlooked algebraic and geometric structures associated with these vector spaces.



\tableofcontents

\section{Introduction}
We begin with some basic notation. Let $\N$, $\R$ and $\CC$ be the set of positive integers, the set of real numbers and complex numbers, respectively, and $\N_0:=\N\cup \{0\}$. For a finite sequence ${\bfk} =\bfk_r:= (k_1,\ldots, k_r)$ of positive integers, we put
\begin{equation*}
|{\bfk}|:=k_1+\cdots+k_r,\quad \dep({\bfk}):=r,
\end{equation*}
and call them the \emph{weight} and the \emph{depth} of ${\bfk}$, respectively. If $k_1>1$, ${\bfk}$ is called \emph{admissible}.

\subsection{Multiple mixed values and their alternating version}

In \cite{XuZhao2020a}, we define the \emph{multiple mixed values} (or \emph{multiple $M$-values}, MMVs) for an admissible $\bfk=(k_1,\ldots,k_r)$ and $\bfeps=(\eps_1,\ldots,\eps_r)\in\{\pm1\}^r$ by
\begin{align*}
M^\bfeps(k_1,\ldots,k_r)&:=\sum_{m_1>\cdots>m_r>0} \frac{(1+\eps_1(-1)^{m_1})\cdots (1+\eps_r(-1)^{m_r})}{m_1^{k_1}\cdots m_r^{k_r}}\in \R\\
={}&\int_0^1 \om_0^{k_1-1}\om_{\eps_1\eps_2}\om_0^{k_2-1}\om_{\eps_2\eps_3}\cdots \om_0^{k_{r-1}-1}\om_{\eps_{r-1}\eps_r}\om_0^{k_r-1}\om_{\eps_r},
\end{align*}
where
\begin{equation*}
\om_0:=\frac{dt}{t},\quad \om_{-1}:=\frac{2dt}{1-t^2},\quad \om_1:=\frac{2tdt}{1-t^2}.
\end{equation*}
The theory of iterated integrals was developed first by K.T. Chen in the 1960's. It has played important roles in the study of algebraic topology and algebraic geometry  since then. Its simplest form is
\begin{align*}
\int_0^1 f_1(t)dtf_{2}(t)dt\cdots f_p(t)dt:={}&\, \int\limits_{1>t_1>\cdots>t_p>0}f_1(t_1)f_{2}(t_{2})\cdots f_p(t_p)dt_1dt_2\cdots dt_p.
\end{align*}
One can extend these to iterated integrals over any piecewise smooth path on the complex plane via pull-backs.
We refer the interested reader to Chen's original work \cite{KTChen1977} for more details.

We now recall the definition of a few classical objects. For admissible $\bfk:=(k_1,\ldots,k_r)\in \N^r$,
the \emph{multiple zeta values} (MZVs) are defined by (cf. \cite{H1992,DZ1994})
\begin{align*}
\zeta(\bfk)\equiv \zeta(k_1,\ldots,k_r):=\sum_{n_1>\cdots>n_r\geq 1}
\frac{1}{n_1^{k_1}\cdots n_r^{k_r}}\in \R,
\end{align*}
Hoffman's \emph{multiple $t$-values} (MtVs) are defined by (see \cite{H2019})
\begin{align*}
t(\bfk):=\sum_{n_1>\cdots>n_r>0\atop n_i\ \text{odd}} \frac{1}{n_1^{k_1}\cdots n_r^{k_r}}
={}&\sum_{n_1>\cdots>n_r>0} \frac{1}{(2n_1-1)^{k_1}(2n_2-1)^{k_2}\cdots (2n_r-1)^{k_r}},
\end{align*}
the \emph{multiple $S$-values} (MSVs) are defined by (see \cite{XuZhao2020a})
\begin{align*}
S(\bfk):=\sum_{n_1>\cdots>n_r>0\atop n_i\equiv r-i\ (\text{mod}\ 2)} \frac{2^r}{n_1^{k_1}\cdots n_r^{k_r}}
={}&\sum_{n_1>\cdots>n_r>0} \frac{2^r}{(2n_1-r+1)^{k_1}(2n_2-r+2)^{k_1}\cdots (2n_r)^{k_r}},
\end{align*}
and the Kaneko-Tsumura \emph{multiple $T$-values} (MTVs) are defined by (see \cite{KanekoTs2019})
\begin{align*}
T(\bfk):=\sum_{n_1>\cdots>n_r>0\atop n_i\equiv r-i+1\ (\text{mod}\ 2)} \frac{2^r}{n_1^{k_1}\cdots n_r^{k_r}}
={}&\sum_{n_1>\cdots>n_r>0} \frac{2^r}{(2n_1-r)^{k_1}(2n_2-r+1)^{k_2}\cdots (2n_r-1)^{k_r}}.
\end{align*}
As a normalized version of the MtVs, we call
\begin{align*}
\t(\bfk):=\sum_{n_1>n_2>\cdots>n_r\geq 1}\frac{1}{(n_1-1/2)^{k_1}(n_2-1/2)^{k_2}\cdots(n_r-1/2)^{k_r}}
={}&2^{|\bfk|}t(\bfk)
\end{align*}
the \emph{multiple $\t$-values}. According to the definitions, $\t(k)=2^st(k)=\ze_H(k;1/2)=(2^k-1)\zeta(k)$ for any integer $k\geq 2$, where $\ze_H(s;a)$ is the \emph{Hurwitz zeta function} and $\zeta(s)$ is the \emph{Riemann zeta function}.

We now list some important special cases of the MMVs.
\begin{itemize}

\item MMVs satisfy the series stuffle relations and integral shuffle
relations;

\item All $\eps_j=1$ in MMVs $M^\bfeps(\bfk)$ $\Longrightarrow$ MZVs$\times \frac{1}{2^{|\bfk|-r}}$;

\item All $\eps_j=-1$ in MMVs $M^\bfeps(\bfk)$ $\Longrightarrow$ MtVs$\times 2^r$;

\item All $\eps_j=(-1)^{r-j}$ in MMVs $M^\bfeps(\bfk)$ $\Longrightarrow$ MSVs;

\item All $\eps_j=(-1)^{r+1-j}$ in MMVs $M^\bfeps(\bfk)$ $\Longrightarrow$ MTVs.

\end{itemize}

The systematic study of MZVs began with the works of Hoffman \cite{H1992} and Zagier \cite{DZ1994} in the early 1990s. Due to their surprising and sometimes mysterious appearance in the study of many branches of mathematics and theoretical physics, these special values have attracted a lot of attention and interest in the past three decades (for example, see the book by the third author \cite{Zhao2016}).
For Hoffman's MtVs and Kaneko-Tsumura's MTVs, a number of mathematicians also studied their various formulas similar to MZVs by applying a variety of methods. In  \cite{LX2020,Takeyama2021,Z2015} (weighted)-sum formulas were studied,
in  \cite{BCJXXZ,KanekoTs2022} double and triple $T$-values were studied
(we will consider triple AMMV's in a broader context in \S 4). Some explicit evaluations of more general (alternating) triple $T$-values and (alternating) triple $t$-values are given in \cite{XuWang2020}. A symmetry result for MtV's was shown in \cite{ChaHpoff2022} and certain special values of MtV's were given in  \cite{Cha2022,Cha2022-1,Mura21,XuZhao2023Aug} some of
which were applied to show results concerning the motivic basis and motivic descent.

All the above variants of MZVs are level two cases of the objects considered by Yuan and the third author in \cite[(2.1)]{YuanZh2014a} where they define more generally the MZVs of level $N$.
In general, let $\bfk=(k_1,\ldots,k_r)\in\N^r$ and $\bfz=(z_1,\dotsc,z_r)$, where $z_1,\dotsc,z_r$ are $N$th roots of unity. We can define the \emph{colored MZVs} (CMZVs) of level $N$ by
\begin{equation}\label{equ:defnMPL}
\Li_{\bfk}(\bfz):=\sum_{n_1>\cdots>n_r>0}
\frac{z_1^{n_1}\dots z_r^{n_r}}{n_1^{k_1} \dots n_r^{k_r}}\in \CC,
\end{equation}
which converge if $(k_1,z_1)\ne (1,1)$ (see \cite{YuanZh2014a} and \cite[Ch. 15]{Zhao2016}), in which case we call $({\bfk};\bfz)$ \emph{admissible}. The level two colored MZVs are often called \emph{Euler sums} or \emph{alternating MZVs}. In this case, namely,
when $(z_1,\dotsc,z_r)\in\{\pm 1\}^r$ and $(k_1,z_1)\ne (1,1)$, we set
$\zeta(\bfk;\bfz)= \Li_\bfk(\bfz)$. Further, we put a bar on top of
$k_{j}$ if $z_{j}=-1$. For example,
\begin{equation*}
\zeta(\bar2,3,\bar1,4)=\zeta(2,3,1,4;-1,1,-1,1).
\end{equation*}
More generally, for any $(k_1,\dotsc,k_r)\in\N^r$, the \emph{classical multiple polylogarithm function} (MPL) with $r$-variables is defined by
\begin{align*}
\Li_{k_1,\dotsc,k_r}(x_1,\dotsc,x_r):=\sum_{n_1>n_2>\cdots>n_r>0} \frac{x_1^{n_1}\dotsm x_r^{n_r}}{n_1^{k_1}\dotsm n_r^{k_r}}
\end{align*}
which converges if $|x_1\cdots x_j|<1$ for all $j=1,\dotsc,r$. It can be analytically continued to a multi-valued
meromorphic function on $\CC^r$ (see \cite{Zhao2007d}). In particular, if $x_1=x,x_2=\cdots=x_r=1$, then 
$\Li_{k_1,\ldots,k_r}(x,\{1\}_{r-1})$ is the \emph{single-variable multiple polylogarithm function}, also 
called \emph{generalized polylogarithm} by some authors. Here $\{S\}_m$ denotes the sequence obtained by 
repeating the string $S$ exactly $m$ times. To save space, for a number such as 1 we also write $1_m=\{1\}_m$.

We now consider the MMVs in some more details.
By abuse of notation it is sometimes more transparent to write
$\eps=\ev$ if $\eps=1$ and $\eps=\od$ if $\eps=-1$. To save space, we may put a check on top of $s_j$ if and only if $\eps_j=\od$. For example,
\begin{equation*}
M(\wc{s_1},s_2)=M^{\od\ev}(s_1,s_2)=\sum_{\substack{m_1>m_2>0\\ m_1\in\od,m_2\in\ev}} \frac{4}{m_1^{s_1} m_2^{s_2}}.
\end{equation*}

To extend the ideas in \cite{XuZhao2020b}, we can also define the \emph{alternating multiple mixed values} (AMMVs) for any $\bfs=(s_1,\ldots,s_r)\in\N^r$, $\bfeps=(\eps_1, \dots, \eps_r)\in\{\pm 1\}^r$, and $\bfgs=(\sigma_1, \dots, \sigma_r)\in\{\pm 1\}^r$ by
\begin{equation*}
M_\bfgs^\bfeps(\bfs):=\sum_{m_1>\cdots>m_r>0} \frac{(1+\eps_1(-1)^{m_1})\sigma_1^{(2m_1+1-\eps_1)/4} \cdots (1+\eps_r(-1)^{m_r})\sigma_r^{(2m_r+1-\eps_r)/4}}{m_1^{s_1} \cdots m_r^{s_r}}
\end{equation*}
which converge if and only if $(s_1,\sigma_1)\neq (1,1)$.
As usual, we call $s_1+\cdots+s_r$ and $r$ the \emph{weight} and \emph{depth}, respectively. We adopt the same convention
for Euler sums by putting a bar on top of $s_j$ if $\sigma_j=-1$, and then set $\sigma_j=\pm$ if $\sigma_j=\pm1$. For example,
\begin{align*}
M(\wc{\bar{a}},\bar{b},\wc{c},d)={}&\, M_{--++}^{\od\ev\od\ev}(a,b,c,d)
=\sum_{\substack{m_1>m_2>m_3>m_4>0\\ m_1,m_3\in\od,\ m_2,m_4\in\ev}}
\frac{16(-1)^{(m_1+1)/2} (-1)^{m_2/2}}{m_1^a m_2^b m_3^c m_4^d}.
\end{align*}
It is not too hard to see that the AMMVs is intimately related to level four CMZVs.

We now turn to the iterated integrals of AMMVs. Set
\begin{equation}\label{equ:1-forms}
\om_0:=\frac{dt}{t},
\quad \om_{+1}^{-1}:=\frac{2dt}{1-t^2},
\quad \om_{-1}^{-1}:=\frac{-2dt}{1+t^2},
\quad \om_{+1}^{+1}:=\frac{2tdt}{1-t^2},
\quad \om_{-1}^{+1}:=\frac{-2tdt}{1+t^2},
\end{equation}
and
\begin{equation}\label{equ:Wdefn}
W_{\sigma}^{\eps_1,\eps_2}:=
\left\{
  \begin{array}{ll}
  -\om_{\sigma}^{\eps_1\eps_2},\quad\  &  \hbox{if $\sigma=\eps_2=-\eps_1=-1$;} \\
  \om_{\sigma}^{\eps_1\eps_2}, \quad & \hbox{otherwise.}
  \end{array}
\right.
\end{equation}
It is straight-forward to deduce that AMMVs can be expressed by the following iterated integrals
\begin{equation}\label{AMMV-iterated-integrals}
M_\bfgs^{\bfeps}(\bfs)=
\int_0^1 \om_0^{s_1-1}W_{\sigma_1}^{\eps_1,\eps_2}
\cdots
\om_0^{s_{r-1}-1}W_{\sigma_1\sigma_2\cdots\sigma_{r-1}}^{\eps_{r-1},\eps_r}
\om_0^{s_r-1}\om_{\sigma_1\sigma_2\cdots\sigma_r}^{\eps_r}.
\end{equation}
For example,
\begin{align*}
M(\wc{3},\bar2)=M_{+-}^{\od\ev}(3,2)=\sum_{n_1>n_2>0}\frac{4(-1)^{n_2}}{(2n_1-1)^3(2n_2)^{2}}
=\int_{0}^{1}\om_0^2\om_{+1}^{-1} \om_0\om_{-1}^{+1}
\end{align*}
and
\begin{align*}
M(\bar2,3,\wc{\bar4})
=\sum_{n_1>n_2>n_3>0}\frac{8(-1)^{n_1+n_3-1}}
{(2n_1-2)^{2}(2n_2-2)^{3}(2n_3-1)^{4}}
=\int_{0}^{1} \om_0\om_{-1}^{+1}\om_0^{2}(-\om_{-1}^{-1})\om_0^{3}\om_{+1}^{-1}.
\end{align*}
For $\bfgs,\bfeps\in\{\pm 1\}^r$ define $\eta_j:=\prod_{n=j}^r \eps_n$ for $j=1,\dots,r$ and
\begin{equation}\label{equ:sgn_bfgs_bfeps}
\sgn(\bfgs,\bfeps):=(-1)^{\sharp\{l<r \mid \sigma_{l}=\eps_{l}=\eta_{l+1}=-1\}}.
\end{equation}
Then for all $\bfk=(k_1,\dots,k_r)\in\N^r$ with $(k_1,\sigma_1)\ne(1,1)$, we have
\begin{align}\label{AMMV-iterated-integrals-duality}
\int_0^1 \om_0^{k_1-1}\om_{\sigma_1}^{\eps_1}\cdots \om_0^{k_r-1}\om_{\sigma_r}^{\eps_r}
=\sgn(\bfgs,\bfeps)M_{\, \sigma_1, \sigma_2\sigma_1,\ldots, \sigma_r\sigma_{r-1}}
^{\eta_1, \ \ \eta_2\ ,\ldots,\ \eta_r\ }(\bfk).
\end{align}
For example, if $\bfeps=(1,-1,-1,1)$ and $\bfgs=(-1,-1,-1,1)$ then $\sgn(\bfgs,\bfeps)=-1$ and
we see that
\begin{align*}
-\int_{0}^{1}
\om_0^{k_1-1}\om_{-1}^{+1}\om_0^{k_2-1}\om_{-1}^{-1}\om_0^{k_3-1}\om_{-1}^{-1}\om_0^{k_4-1}\om_{+1}^{+1}
=-M_{-++-}^{\ev\od\ev\ev}(k_1,k_2,k_3,k_4).
\end{align*}
By Chen's theory of iterated integrals (see e.g., \cite[Ch 4]{Zhao2016}) we know that
AMMVs must satisfy the integral shuffle relations. For example,
\begin{align} \label{equ:finDBSF1}
 M(\cbar1)M(\cbar2)
=\int_{0}^{1} \om_{-1}^{-1} \int_{0}^{1} \om_0\om_{-1}^{-1}
=\int_{0}^{1} \om_{-1}^{-1}\om_0\om_{-1}^{-1}+2\int_{0}^{1} \om_0\om_{-1}^{-1}\om_{-1}^{-1}
={}&-M(\bar1,\wc2)-2M(\bar2,\wc1)
\end{align}
and
\begin{multline*}
 M(\bar3)M(\wc{2})= M_{-}^{\ev}(3)M_{+}^{\od}(2)=\int_{0}^{1} \om_0^2\om_{-1}^{+1} \int_{0}^{1} \om_0\om_{+1}^{-1}\\
=\!\int_{0}^{1} \om_0^2\om_{-1}^{+1}\om_0\om_{+1}^{-1}+\int_{0}^{1} \om_0\om_{+1}^{-1}\om_0^2\om_{-1}^{+1}+2\int_{0}^{1} \om_0^2\om_{+1}^{-1}\om_0\om_{-1}^{+1} % \\&\quad
+3\int_{0}^{1} \om_0^3\om_{+1}^{-1}\om_{-1}^{+1}+3\int_{0}^{1} \om_0^3\om_{-1}^{+1}\om_{+1}^{-1}\\
=M(\cbar3,\cbar2)+M(\wc2,\bar3)+2M(\wc3,\bar2)+3M(\wc4,\bar1)+3M(\cbar4,\cbar1).
\end{multline*}
Moreover, by considering how two sets of summation indices $m_1>\cdots>m_r$ and $m_1'>\cdots>m_r'$ interleave
(also taking into account their parity restrictions), it is clear that AMMVs also satisfy the series stuffle relations. For examples,
\begin{align}
M(\cbar1)M(\cbar2)={}&M(\cbar1,\cbar2)+M(\cbar2,\cbar1)+2M(\wc3), \label{equ:finDBSF2}\\
M(\bar2,3,\cbar4)M(\cbar2)={}&M(\bar2,3,\cbar4,\cbar2)+M(\bar2,3,\cbar2,\cbar4)
+M(\bar2,\cbar2,3,\cbar4)+M(\cbar2,\bar2,3,\cbar4)+2M(\bar2,3,\wc6),\nonumber\\
M(\bar1,\cbar2)M(3,\cbar2)={}&M(\bar1,\cbar2,3,\cbar2)+M(3,\cbar2,\bar1,\cbar2)
+M(\bar1,3,\cbar2,\cbar2)+M(3,\bar1,\cbar2,\cbar2)\nonumber\\
&+2M(\bar1,3,\wc4)+2M(3,\bar1,\wc4)+2M(\bar4,\cbar2,\cbar2)+4M(\bar4,\wc4).\nonumber
\end{align}
Note that stuffing can happen only when two components from two indices correspond to two equal $\eps$'s. For each incidence of such stuffings, the two signs (i.e., $\sigma$'s) corresponding to the two index components are multiplied and an extra factor of 2 is produced.

Combining \eqref{equ:finDBSF1} and \eqref{equ:finDBSF2} we obtain a so-called finite double shuffle relation:
\begin{equation}\label{equ:finDBSFeg}
M(\cbar1,\cbar2)+M(\cbar2,\cbar1)+2M(\wc3)+M(\bar1,\wc2)+2M(\bar2,\wc1)=0.
\end{equation}
It is one of the main goals of this paper to find as many possible $\Q$-linear relations among AMMVs as possible.
From the theory of MZVs and more generally CMZVs, we know it is insufficient to consider only convergent AMMVs,
which motivates us to study their regularization.

\subsection{Alternating multiple $S$/$T$-harmonic sums}

For positive integers $r$ and $n$ such that $n\ge r$, we define the following subsets
of $\N^m$:
\begin{align*}
&D_{n,r} :=
\left\{
\begin{array}{ll}
\Big\{(n_1,\dotsc,n_r)\in\N^{r} \mid n\geq n_1 >n_2\geq \cdots \geq n_{r-2}>n_{r-1}\geq n_r>0 \Big\},\phantom{\frac12}\ & \hbox{if $2\nmid r$;} \\
\Big\{(n_1,\dotsc,n_r)\in\N^{r} \mid n> n_1 \geq n_2> \cdots \geq n_{r-2}>n_{r-1}\geq n_r>0 \Big\},\phantom{\frac12}\ & \hbox{if $2\mid r$,}
\end{array}
\right.  \\
&E_{n,r} :=
\left\{
\begin{array}{ll}
\Big\{(n_1,\dotsc,n_r)\in\N^{r}\mid n> n_1 \geq n_2>\cdots > n_{r-2}\geq n_{r-1}>n_r>0 \Big\},\phantom{\frac12}\ & \hbox{if $2\nmid r$;} \\
\Big\{(n_1,\dotsc,n_r)\in\N^{r}\mid n\geq n_1 >n_2\geq \cdots > n_{r-2}\geq n_{r-1}>n_r>0  \Big\}, \phantom{\frac12}\ & \hbox{if $2\mid r$.}
\end{array}
\right.
\end{align*}

\begin{defn} For positive integer $m$ and $\bfgs_{r}:=(\sigma_1,\sigma_2,\ldots,\sigma_r)\in \{\pm 1\}^r$, define
\begin{align}
&S^{\bfgs_{2m-1}}_n({\bfk_{2m-1}}):= \sum_{\bfn\in E_{n,2m-1}} \frac{2^{2m-1}\sigma_1^{n_1}\sigma_2^{n_2}\cdots \sigma_{2m-1}^{n_{2m-1}}}{(\prod_{j=1}^{m-1} (2n_{2j-1})^{k_{2j-1}}(2n_{2j}-1)^{k_{2j}})(2n_{2m-1})^{k_{2m-1}}},\label{AMOS}\\
&S^{\bfgs_{2m}}_n({\bfk_{2m}}):= \sum_{\bfn\in E_{n,2m}} \frac{2^{2m}\sigma_1^{n_1}\sigma_2^{n_2}\cdots \sigma_{2m}^{n_{2m}}}{\prod_{j=1}^{m} (2n_{2j-1}-1)^{k_{2j-1}}(2n_{2j})^{k_{2j}}},\label{AMES}\\
&T^{\bfgs_{2m-1}}_n({\bfk_{2m-1}}):= \sum_{\bfn\in D_{n,2m-1}} \frac{2^{2m-1}\sigma_1^{n_1}\sigma_2^{n_2}\cdots  \sigma_{2m-1}^{n_{2m-1}}}{(\prod_{j=1}^{m-1} (2n_{2j-1}-1)^{k_{2j-1}}(2n_{2j})^{k_{2j}})(2n_{2m-1}-1)^{k_{2m-1}}},\label{AMOT}\\
&T^{\bfgs_{2m}}_n({\bfk_{2m}}):= \sum_{\bfn\in D_{n,2m}} \frac{2^{2m}\sigma_1^{n_1}\sigma_2^{n_2}\cdots \sigma_{2m}^{n_{2m}}}{\prod_{j=1}^{m} (2n_{2j-1})^{k_{2j-1}}(2n_{2j}-1)^{k_{2j}}},\label{AMET}
\end{align}
where $T^{\bfgs_{2m}}_n({\bfk_{2m}})=S^{\bfgs_{2m-1}}_n({\bfk_{2m-1}})=S^{\bfgs_{2m}}_n({\bfk_{2m}}):=0$ if $n\leq m$, 
and $T^{\bfgs_{2m-1}}_n({\bfk_{2m-1}}):=0$ if $n<m$. Moreover, we set $S^{\bfgs}_n(\emptyset)=T^{\bfgs}_n(\emptyset):=1$ 
for convenience. We call \eqref{AMOS} and \eqref{AMES} \emph{alternating multiple $S$-harmonic sums},
and call \eqref{AMOT} and \eqref{AMET} \emph{alternating multiple $T$-harmonic sums}.
\end{defn}

According to the definitions of alternating multiple $S$-harmonic sums and alternating multiple $T$-harmonic sums, we have the following relations
\begin{align}
\label{relations-odd-alternatingmultipleS-harmonicsums}
&S^{\sigma_1,\sigma_2,\ldots,\sigma_{2m-1}}_n(k_1,k_2,\ldots,k_{2m-1})=2\sum_{j=1}^{n-1} \frac{S^{\sigma_2,\sigma_3,\ldots,\sigma_{2m-1}}_j(k_2,k_3,\ldots,k_{2m-1})}{(2j)^{k_{1}}}\sigma_1^j,\\
\label{relations-even-alternatingmultipleS-harmonicsums}
&S^{\sigma_1,\sigma_2,\ldots,\sigma_{2m}}_n(k_1,k_2,\ldots,k_{2m})=2\sum_{j=1}^{n} \frac{S^{\sigma_2,\sigma_3,\ldots,\sigma_{2m}}_j(k_2,k_3,\ldots,k_{2m})}{(2j-1)^{k_{1}}}\sigma_1^j,\\
\label{relations-odd-alternatingmultipleT-harmonicsums}
&T^{\sigma_1,\sigma_2,\ldots,\sigma_{2m-1}}_n(k_1,k_2,\ldots,k_{2m-1})=2\sum_{j=1}^{n} \frac{T^{\sigma_2,\sigma_3,\ldots,\sigma_{2m-1}}_j(k_2,k_3,\ldots,k_{2m-1})}{(2j-1)^{k_{1}}}\sigma_1^j,\\
\label{relations-even-alternatingmultipleT-harmonicsums}
&T^{\sigma_1,\sigma_2,\ldots,\sigma_{2m}}_n(k_1,k_2,\ldots,k_{2m})=2\sum_{j=1}^{n-1} \frac{T^{\sigma_2,\sigma_3,\ldots,\sigma_{2m}}_j(k_2,k_3,\ldots,k_{2m})}{(2j)^{k_{1}}}\sigma_1^j.
\end{align}

By direct calculations we obtain
\begin{align*}
&S^{\sigma_1,\ldots,\sigma_{2m-1}}_n(k_1,k_2,\ldots,k_{2m-1})
=\sigma_1^{m-1}\prod\limits_{j=1}^{m-2} (\sigma_{2j}\sigma_{2j+1})^{m-j-1} \sum_{\substack{n+m>n_1>n_2> \\ \cdots>n_{2m-1}>0}} \frac{2^{2m-1}\sigma_1^{n_1}\sigma_2^{n_2}\cdots \sigma_{2m-1}^{n_{2m-1}}}{\prod\limits_{j=1}^{2m-1} (2n_j-2m+j+1)^{k_j}}\nonumber\\[-.5em]
={}&\sum_{\substack{2n>n_1>n_2>\\ \cdots>n_{2m-1}>0}} \!\!\! \frac{\prod\limits_{j=1}^{m-1}\left(1\!+\!(-1)^{n_{2j-1}}\right)\sigma_{2j-1}^{n_{2j-1}/2}\left(1\!-\!(-1)^{n_{2j}}\right)\sigma_{2j}^{(n_{2j}+1)/2}}{n_1^{k_1}n_2^{k_2}\cdots n_{2m-2}^{k_{2m-2}}n_{2m-1}^{k_{2m-1}}}\cdot \left(1\!+\!(-1)^{n_{2m-1}}\right)\sigma_{2m-1}^{n_{2m-1}/2},\\[1.5em]
&S^{\sigma_1,\ldots,\sigma_{2m}}_n(k_1,k_2,\ldots,k_{2m})
=\prod\limits_{j=1}^{m-1} (\sigma_{2j-1}\sigma_{2j})^{m-j} \sum_{\substack{n+m>n_1>n_2> \\ \cdots>n_{2m}>0}} \frac{2^{2m}\sigma_1^{n_1}\sigma_2^{n_2}\cdots \sigma_{2m}^{n_{2m}}}{\prod\limits_{j=1}^{2m} (2n_j-2m+j)^{k_j}}\nonumber\\[-.5em]
={}&\sum_{\substack{2n>n_1>n_2>\\ \cdots>n_{2m}>0}} \!\!\! \frac{\prod\limits_{j=1}^{m}\left(1\!-\!(-1)^{n_{2j-1}}\right)\sigma_{2j-1}^{(n_{2j-1}+1)/2}\left(1\!+\!(-1)^{n_{2j}}\right)\sigma_{2j}^{n_{2j}/2}}{n_1^{k_1}n_2^{k_2}\cdots n_{2m-1}^{k_{2m-1}}n_{2m}^{k_{2m}}},\\[1.5em]
&T^{\sigma_1,\ldots,\sigma_{2m-1}}_n(k_1,k_2,\ldots,k_{2m-1})
=\prod\limits_{j=1}^{m-1} (\sigma_{2j-1}\sigma_{2j})^{m-j} \sum_{\substack{n+m>n_1>n_2> \\ \cdots>n_{2m-1}>0}} \frac{2^{2m-1}\sigma_1^{n_1}\sigma_2^{n_2}\cdots \sigma_{2m-1}^{n_{2m-1}}}{\prod\limits_{j=1}^{2m-1} (2n_j-2m+j)^{k_j}}\nonumber\\[-.5em]
={}&\sum_{\substack{2n>n_1>n_2>\\ \cdots>n_{2m-1}>0}} \!\!\! \frac{\prod\limits_{j=1}^{m-1}\left(1\!-\!(-1)^{n_{2j-1}}\right)
\sigma_{2j-1}^{(n_{2j-1}+1)/2}\left(1\!+\!(-1)^{n_{2j}}\right)\sigma_{2j}^{n_{2j}/2}}{n_1^{k_1}n_2^{k_2}\cdots n_{2m-2}^{k_{2m-2}}n_{2m-1}^{k_{2m-1}}}\cdot \left(1\!-\!(-1)^{n_{2m-1}}\right)\sigma_{2m-1}^{(n_{2m-1}+1)/2},\\[1.5em]
&T^{\sigma_1,\ldots,\sigma_{2m}}_n(k_1,k_2,\ldots,k_{2m})
=\prod\limits_{j=1}^{m-1} (\sigma_{2j-1})^{m-j+1}(\sigma_{2j})^{m-j} \sum_{\substack{n+m>n_1>n_2> \\ \cdots>n_{2m}>0}} \frac{2^{2m}\sigma_1^{n_1}\sigma_2^{n_2}\cdots \sigma_{2m}^{n_{2m}}}{\prod\limits_{j=1}^{2m} (2n_j-2m-1+j)^{k_j}}\nonumber\\[-.5em]
={}&\sum_{\substack{2n>n_1>n_2>\\ \cdots>n_{2m}>0}} \!\!\! \frac{\prod\limits_{j=1}^{m}\left(1\!+\!(-1)^{n_{2j-1}}\right)\sigma_{2j-1}^{n_{2j-1}/2}\left(1\!-\!(-1)^{n_{2j}}\right)\sigma_{2j}^{(n_{2j}+1)/2}}{n_1^{k_1}n_2^{k_2}\cdots n_{2m-1}^{k_{2m-1}}n_{2m}^{k_{2m}}}.
\end{align*}
Hence, for any $\bfk=(k_1,\ldots,k_r)\in\N^r$ and $\bfgs:=(\sigma_1,\ldots,\sigma_r)\in \{\pm 1\}^r$
with $(s_1,\sigma_1)\neq (1,1)$, we may define the \emph{alternating MSVs} (AMSVs) and \emph{alternating MTVs} (AMTVs) by
\begin{equation}\label{defn-AMTVs}
S^{\bfgs}(\bfk):=\lim_{n\rightarrow \infty }S_n^{\bfgs}(\bfk), \qquad
T^{\bfgs}(\bfk):=\lim_{n\rightarrow \infty }T_n^{\bfgs}(\bfk).
\end{equation}
We may compactly indicate the presence of an alternating sign as before:
if $\sigma_j=-1$ then we place a bar over the corresponding component $k_j$. For example,
\begin{equation*}
S(\bar 1,3,\bar 2,4,\bar 2)=S^{-1,1,-1,1,-1}(1,3,2,4,2)\quad \text{and}\quad T(\bar 2,3,\bar 1,4)=T^{-1,1,-1,1}(2,3,1,4).
\end{equation*}
Obviously, AMMVs can have, similar to MMVs, the following four special cases
\begin{itemize}

\item All $\eps_j=1$ in AMMVs $M_\bfgs^{\bfeps}(\bfk)$ $\Longrightarrow$ AMZVs$\times \frac{1}{2^{k_1+k_2+\cdots+k_r-r}}$;

\item All $\eps_j=-1$ in AMMVs $M_\bfgs^{\bfeps}(\bfk)$ $\Longrightarrow$ AMtVs$\times 2^r$;

\item All $\eps_j=(-1)^{r-j}$ in AMMVs $M_\bfgs^{\bfeps}(\bfk)$ $\Longrightarrow$ AMSVs;

\item All $\eps_j=(-1)^{r+1-j}$ in AMMVs $M_\bfgs^{\bfeps}(\bfk)$ $\Longrightarrow$ AMTVs.
\end{itemize}
For the detailed definition and introduction of \emph{alternating MZVs} (AMZVs) and \emph{alternating MtVs} (AMtVs), 
see \cite{H2019,Zhao2016}. Let $\AMZV_w$ be the $\Q$-span of all the AMZVs of weight $w$. 
Setting $\dim_\Q \AMZV_0=1$, we have the dimension bound $\dim_\Q \AMZV_w\le F_w$ obtained by Deligne and Goncharov \cite{DeligneGo2005} 
and moreover a set of generators of $\AMZV_w$ shown by Deligne \cite[Thm. 7.2]{Deligne2010}, where $F_0=F_1=1$ and 
$F_n=F_{n-1}+F_{n-2}$ for all $n\ge 2$. In particular, in a recent paper, S. Charlton \cite{Cha2022} proved the 
AMtVs $t(\{\bar1\}_a,1,\{\bar1\}_b)\ (a,b\in\N)$ can be expressed in terms of $\log2$, Riemann zeta values and 
Dirichlet beta values as defined by \eqref{defn-Dbeta-function}.

According to the definitions of AMSVs and AMTVs, for any $r\in\N$, we have the following iterated integral expressions:
\begin{align*}
&S^{\sigma_1,\sigma_2,\ldots,\sigma_r}(k_1,k_2,\ldots,k_r)=M_{\sigma_1,\sigma_2,\ldots,\sigma_r}^{\ \ldots\ \od\ev\od\ev}(k_1,k_2,\ldots,k_r)\\
={}&\left\{\prod_{j=1}^{[(r-1)/2]} (\sigma_1\cdots \sigma_{2j-(1-\pari(r))/2})\right\}\int_0^1 \om_0^{k_1-1}\om_{\sigma_1}^{-1}\om_0^{k_2-1}\om_{\sigma_1\sigma_2}^{-1}\cdots \om_0^{k_{r-1}-1}\om_{\sigma_1\cdots\sigma_{r-1}}^{-1} \om_0^{k_r-1}\om_{\sigma_1\cdots\sigma_r}^{+1},\\[1em]
&T^{\sigma_1,\sigma_2,\ldots,\sigma_r}(k_1,k_2,\ldots,k_r)=M_{\sigma_1,\sigma_2,\ldots,\sigma_r}^{\ \ldots\ \ev\od\ev\od}(k_1,k_2,\ldots,k_r)\\
={}&\left\{\prod_{j=1}^{[r/2]} (\sigma_1\cdots \sigma_{2j-(1+\pari(r))/2})\right\}\int_0^1 \om_0^{k_1-1}\om_{\sigma_1}^{-1}\om_0^{k_2-1}\om_{\sigma_1\sigma_2}^{-1}\cdots \om_0^{k_r-1}\om_{\sigma_1\cdots\sigma_r}^{-1},
\end{align*}
where $\pari(r)=-1$ if $r$ is odd and $\pari(r)=1$ if $r$ is even.

\subsection{Main results}

The primary goals of this paper are to study the explicit relations of AMMVs and some of their special types, and
establish a few explicit evaluations of the AMMVs and related values via AMMVs of lower depths.

The remainder of this paper is organized as follows.

In Sections \ref{sec:AMMV} and \ref{RAMMVs} we find the series stuffle relations and integral shuffle relations of AMMVs, and set up the algebraic framework for the regularized double shuffle relations (DBSFs) of AMMVs. As an important application, we confirm our previous parity conjecture \cite[Conj. 5.2]{XuZhao2020b}
on AMMVs in Theorem~\ref{thm:MMVparity}. Moreover, we give some specific examples to illustrate this phenomenon.

In Section \ref{AMTVs-AMSVs}, we first prove four integral identities involving the arctangent function. Then we apply these formulas obtained to establish two explicit relations between AMSVs and AMTVs.

In Section \ref{DCAMMVs}, we compute the dimensions of AMMVs, AMtVs, AMSVs and AMTVs for weight less than 9. Supported by both theoretical and numerical evidence, we also formulate a few conjectures concerning the exact dimensions of the above-mentioned subspaces of AMMVs. We end this paper with two conjectural diagrams showing the relations among the various vector spaces appearing in this paper.

The following diagram shows some combinatorial relations among CMZV, AMMV, AMTV, AMSV, etc.
\begin{center}
\includegraphics[height=2.8 in]{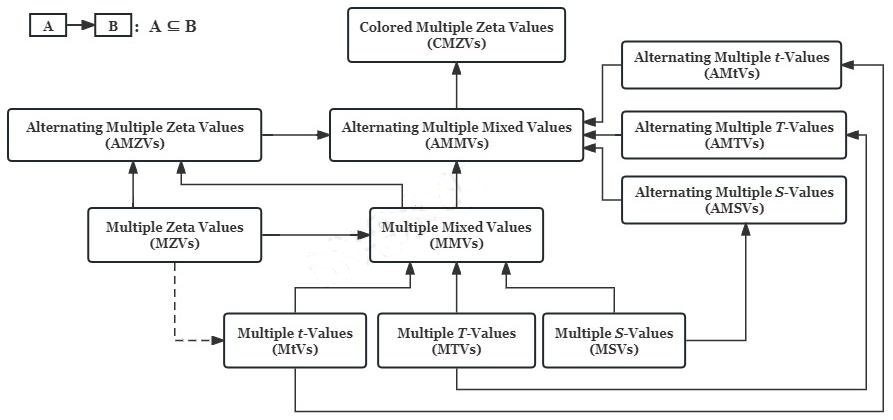}
\end{center}

As pointed out by the referee, there is a highly nontrivial inclusion on the motivic level denoted by the dashed curve
at the lower left corner of the diagram: every motivic MZV is a $\Q$-linear combination of motivic MtVs. In fact,
Murakami even showed in \cite{Mura21} that the set of MtVs $\{t(k_1,\dots, k_p) | k_1, \dots, k_p\in\{2, 3\}\}$
generates all of MZVs over $\Q$.

\section{Algebraic setup for AMMVs and finite double shuffle relations}\label{sec:AMMV}
We know that the AMMVs satisfy the series stuffle relations.
To define the stuffle relations of alternating multiple harmonic sums a double cover $\db$ of $\N$ is defined in \cite[Chapter 7]{Zhao2016}: $\db=\N\cup \bar\N=\N\cup \{\bar{n}: n\in\N\}$ which is equipped with a binary operation $\oplus$ such that $a\oplus b= \bar{a}\oplus \bar{b}:= a+b$ and $\bar{a}\oplus b=a\oplus \bar{b}:= \ol{a+b}$ for all  $a,b\in \N$.

We now define a double cover $\dk$ of $\db$: $\dk=\db\cup \wc{\db}=\db\cup\{\wc{s}: s\in\db\}$ together with a map extending the binary operation $\oplus$, still denoted by $\oplus$, such that $\wc{a}\oplus \wc{b}:=\wc{a\oplus b}$ for any $a,b\in \db$.

For every \emph{decorated positive integer} $k\in\dk$ (say $k=s,\bar{s},\wc{s}$ or $k=\wc{\bar{s}}$ for some $s\in\N$),
we define its \emph{absolute value, sign} and \emph{parity} by
\begin{equation*}
|k|=s,
\qquad
\sgn(k)=
\left\{
\begin{array}{ll}
1,  & \quad \hbox{if $k=s$ or $\wc{s}$;} \\
-1,  & \quad \hbox{if $k=\bar{s}$ or $\wc{\bar{s}}$,}
\end{array}
\right.
\qquad
\pari(k)=
\left\{
\begin{array}{ll}
1,  & \quad \hbox{if $k=s$ or $\bar{s}$;} \\
-1, & \quad \hbox{if $k=\wc{s}$ or $\wc{\bar{s}}$,}
\end{array}
\right.
\end{equation*}
respectively. For any $k\in\dk$ define the word of length $|k|$
$$\tz_k:=\tz_{|k|;\sgn(k)}^{\pari(k)}:=\om_0^{|k|-1} \om^{\pari(k)}_{\sgn(k)}.$$
One can now define an algebra of words as follows.
\begin{defn} \label{defn:fAlevel2}
Let $\setX$ be the alphabet consisting of the letters $\om_0$ and $\om_\sigma^\eps$ ($\sigma,\eps=\pm1$). A \emph{word} $\bfw$ is a monomial in the letters in $\setX$.
Its \emph{weight}, denoted by $|\bfw|$, is the number of letters contained in $\bfw$, and its \emph{depth},
denoted by $\dep(\bfw)$, is the number of $\om_\sigma^\eps$'s contained in $\bfw$.
Define the \emph{AMMV} algebra, denoted by $\fA$, to be the (weight) graded
noncommutative polynomial $\Q$-algebra generated by the words (including the empty word $\bfone$) over the alphabet $\setX$.
Let $\fA^0$ be the subalgebra of $\fA$ generated by words not
beginning with $\om_1^{\pm 1}$ and not ending with $\om_0$. The words in $\fA^0$
are called \emph{admissible words.}
\end{defn}

By \eqref{AMMV-iterated-integrals} every AMMV can be expressed as an iterated
integral over the closed interval $[0,1]$ of an admissible word. Thus, for
$\bfw=\tz_\bfk:=\tz_{k_1}\dots \tz_{k_r}\in\fA^0$, we set
\begin{equation*}
\evaM(\bfw):=\int_0^1 \bfw, \qquad M(\bfw):=M(\bfk).
\end{equation*}
We also extend $\evaM$ and $M$ to $\fA^0$ by $\Q$-linearity. Now we set $s_j=|k_j|$,
$\eps_j=\pari(k_j)$ and  $\sigma_j=\sgn(k_j)$ for all $j=1,\dots,r$ and define
\begin{align*}
\bfp (\bfw):={}& \om_0^{s_1-1}W_{\sigma_1}^{\eps_1,\eps_2}
\cdots
\om_0^{s_{r-1}-1}W_{\sigma_1\sigma_2\cdots \sigma_{r-1}}^{\eps_{r-1},\eps_r}
\om_0^{s_r-1}\om_{\sigma_1\sigma_2\cdots \sigma_r}^{\eps_r},\\
\bfq(\bfw):={}& \sgn(\bfgs,\bfeps) \om_0^{s_1-1}\om_{\sigma_1}^{\eta_1}
\om_0^{s_2-1}\om_{\sigma_1\sigma_2}^{\eta_2}
\cdots
\om_0^{s_r-1}\om_{ \sigma_{r-1}\sigma_r}^{\eta_r}
\end{align*}
where $W$ is defined by \eqref{equ:Wdefn}, $\sgn(\bfgs,\bfeps)$ is defined by \eqref{equ:sgn_bfgs_bfeps}, and $\eta_j=\prod_{n=j}^r \eps_n$ as before. Then it is easy to see that
\begin{equation}\label{equ:bfpbfqInv}
 \bfp\circ\bfq = \bfq\circ\bfp ={\rm id}
\end{equation}
is the identity map. From \eqref{AMMV-iterated-integrals} and \eqref{AMMV-iterated-integrals-duality} one has
\begin{equation} \label{equ:MMV2word}
M(\bfw)= \evaM \big(\bfp (\bfw)\big) \quad\text{and}\quad
\evaM(\bfw)= M\big(\bfq(\bfw)\big).
\end{equation}
Here for empty word $\bfone$ we set $M(\bfone)=\evaM(\bfone) =1$ and $\bfp (\bfone)=\bfq (\bfone)=\bfone$.

Let $\fA_{\sha}$ be the algebra of $\fA$ where the
multiplication is defined by the usual shuffle product $\sha$.
Denote the subalgebra $\fA^0$ by $\fA_{\sha}^0$ when one
considers this shuffle product.
\begin{pro} \label{prop:AMMVshahomo}
The map $\evaML:\fA_{\sha}^0\lra \R$ is an algebra homomorphism.
\end{pro}
\begin{proof}
Similarly to the proof of the corresponding result for MZVs, this follows easily from Chen's theory of the shuffle product relations of iterated integrals by formula \eqref{equ:MMV2word}. We leave the details to the interested reader.
\end{proof}

Setting $M(\emptyset)=1$, we easily see that for $d,l\ge 0$, $\bfs=(s_1,\dots,s_r)\in \dk^r$  and $\bft=(t_1,\dots,t_d)\in \dk^d$
\begin{equation*}
M(\bfs)M(\bft)=M(\bfs*\bft),
\end{equation*}
where the commutative stuffle product $\bfs*\bft$ is defined recursively by
\begin{align*}
\emptyset*\bfs=\bfs*\emptyset= & \bfs, \\
(s_1,\dots,s_r)*(t_1,\dots,t_d)= & \big(s_1, (s_2,\dots,s_r)*(t_1,\dots,t_d) \big)
+\big(t_1, (s_1,\dots,s_d)*(t_2,\dots,t_d) \big)\\
&+2\delta(\pari(s_1),\pari(t_1)) \big(s_1\oplus t_1, (s_2,\dots,s_r)*(t_2,\dots,t_d) \big)
\end{align*}
for all $r,d\ge 1$, where the Kronecker symbol $\delta(\eps,\eta)=1$ if $\eps=\eta$ and $0$ otherwise.

We have seen a finite double double shuffle relation in \eqref{equ:finDBSFeg}. As we pointed out, we need to handle
divergent AMMVs appropriately in order to discover more $\Q$-linear relations. In order to do so, we must consider
not only the admissible words but also some non-admissible ones.

\begin{defn}\label{defn:cst}
Denote by $\fA^1$ the subalgebra of $\fA$ generated by words $\tz^\eps_{s;\sigma}$ with $s\in \N$
and $\sigma,\eps=\pm1$. Equivalently, $\fA^1$ is the subalgebra of $\fA$ generated by words not ending with $\om_0$.
We then define another multiplication $\cst$ on $\fA^1$ by
\begin{equation*}
\bfu\cst\bfv=\bfp\big( \bfq(\bfu)*\bfq(\bfv) \big), \qquad \forall \bfu,\bfv\in \fA^1,
\end{equation*}
where $*$ is the ordinary stuffle which is defined by the following:
(i) it distributes over addition, (ii)
$\myone*\bfw=\bfw*\myone=\bfw,\ \forall \bfw\in\fA^1$, and (iii)
$\forall\bfu,\bfv\in\fA^1$, $s,t\in\N$ and $\sigma_1,\sigma_2,\eps_1,\eps_2=\pm 1$
\begin{equation}\label{equ:astDefn}
\tz^{\eps_1}_{s;\sigma_1} \bfu*\tz^{\eps_2}_{t;\sigma_2}\bfv
=\tz^{\eps_1}_{s;\sigma_1}\big(\bfu*\tz^{\eps_2}_{t;\sigma_2}\bfv\big)
+\tz^{\eps_2}_{t;\sigma_2}\big( \tz^{\eps_1}_{s;\sigma_1} \bfu*\bfv\big)
+\delta(\eps_1,\eps_2) \tz^{\eps_1}_{s+t;\sigma_1\sigma_2} (\bfu*\bfv).
\end{equation}
This multiplication $\cst$ is called the \emph{stuffle product} for AMMVs.
\end{defn}

If we denote by $\fA_{\cstt}^1$ the algebra $(\fA^1,\cst)$ then it is not
hard to prove the next proposition.
\begin{pro}\label{prop:AMMVsthomo}
The polynomial algebra $\fA_{\cstt}^1$ is an associative, commutative and weight graded $\Q$-algebra.
\end{pro}
\begin{proof}
One can refer to \cite[Theorem 3.2]{H1997} for a similar proof by using induction on the weight of words.
The key is to use the recursive definition of $\cst$ in Definition~\ref{defn:cst}. Here is a
very sketchy proof of the associativity. Assume $\bfu,\bfv$ and $\bfw$ are three nonempty words (if any one
is the empty word then the associativity is trivial). Then by definition and \eqref{equ:bfpbfqInv}
\begin{equation*}
   (\bfu\cst\bfv)\cst \bfw=\bfp\Big( (\bfq(\bfu)*\bfq(\bfv))*\bfq(\bfw)\Big),\quad
   \bfu\cst(\bfv\cst \bfw)=\bfp\Big( \bfq(\bfu)*(\bfq(\bfv)*\bfq(\bfw))\Big).
\end{equation*}
Thus it suffices to show that $*$ is associative which follows from the recursive definition of \eqref{equ:astDefn} by induction on the total weight of the words $\bfq(\bfu)$, $\bfq(\bfv)$ and $\bfq(\bfw)$. We leave the details to the
interested reader.
\end{proof}

Now we can define the subalgebra $\fA_{\cstt}^0$ similarly to what we did for $\fA_{\sha}^0$
by replacing the shuffle product by the $M$-stuffle product.
\begin{pro} \label{prop:AMMVstIShomo}
The map $\evaML: \fA_{\cstt}^0 \lra \R$ is an algebra homomorphism.
\end{pro}
\begin{proof}
Let $\bfu$ and $\bfv$ be two admissible words. Then by \eqref{equ:bfpbfqInv}
\begin{equation*}
 \evaM\big(\bfu\cst \bfv\big)=M\big(\bfq\circ\bfp\big(\bfq(\bfu)* \bfq(\bfv)\big)\big)
 =  M\big(\bfq(\bfu)* \bfq(\bfv)\big) =M\big(\bfq(\bfu)\big)M\big(\bfq(\bfv)\big)= \evaM(\bfu)\evaM(\bfv)
\end{equation*}
as desired.
\end{proof}

\begin{defn}
Let $w\in\N$ such that $w\ge 2$. For all nontrivial words
$\bf\om_1,\bf\om_2\in \fA^0$ with $|\bf\om_1|+|\bf\om_2|=w$, we say that
the equation
\begin{equation*}
\evaM(\bf\om_1\sha \bf\om_2-\bf\om_1\cst\bf\om_2)=0
\end{equation*}
provides a \emph{finite double shuffle relation} (finite DBSF) of AMMVs of weight $w$.
\end{defn}

See \eqref{equ:finDBSFeg} for a concrete example of this in weight three.

\section{Regularization for divergent AMMVs and double shuffle relations}\label{RAMMVs}
Since there are no weight 1 MZVs (as $\zeta(1)$ diverges), it is impossible to prove the identity
$\zeta(2,1)=\zeta(3)$ via a finite DBSF. However, one can remedy this by considering the
\emph{regularized double shuffle relation} (regularized DBSF) produced by the following mechanism.

First, combining Prop.~\ref{prop:AMMVshahomo} and Prop.~\ref{prop:AMMVstIShomo} and using
the algebra structures of $\fA_{\cstt}^1$ (resp. $\fA_{\sha}^1$) over $\fA_{\cstt}^0$ (resp. $\fA_{\sha}^0$)
we can prove the following algebraic result without too much difficulty.

\begin{pro} \label{prop:extendM*}
We have an algebra homomorphism:
\begin{equation*}
\evaML_{\cstt}: (\fA_{\cstt}^1,\cst)\lra \R[T]
\end{equation*}
which is uniquely determined by the properties that it extends the evaluation map $M:\fA_{\cstt}^0\lra \R$ by sending
$\emz_{1}$ to $T$ and $\emz_{\wc1}$ to $T+2\log 2$. Further, we have an algebra homomorphism:
\begin{equation*}
\evaML_\sha: (\fA_{\sha}^1,\sha)\lra \R[T]
\end{equation*}
which is uniquely determined by the properties that it
extends the evaluation map $\evaML:\fA_\sha^0\lra \R$ by sending
$\emz_{1}$ to $T-\log 2$ and $\emz_{\wc1}$ to $T+\log 2$.
\end{pro}
\begin{proof}
The proof is the same as that for the corresponding results of MMVs.
See \cite[Prop.~2.7 and Prop.~ 2.8]{XuZhao2020a}. The key observations for
the stuffle regularization are
\begin{align*}
  M_N(1)=\,\sum_{k=1}^N \frac{2}{2k}\,=\sum_{k=1}^N \frac{1}{k}= \log N+\gamma+&O(N^{-1} \log N),\\
  M_N(\wc{1})=\sum_{k=1}^N \frac{2}{2k-1}=\sum_{l=1}^{2N} \frac{2}{l}-S_N(1)
  =&2(\log 2N+\gamma)-\log N-\gamma+O(N^{-1} \log N)\\
  =&\log N+\gamma+2\log 2+O(N^{-1} \log N).
\end{align*}
By the usual regularization process we see that $\evaM_{\cstt}(\tz_{1;1}^1)=T$ and $\evaM_{\cstt}(\wc{\tz}_{1;1}^{-1})=T+2\log 2$.

For the shuffle regularization, we take an arbitrarily small $\gl>0$ and find that
\begin{align*}
\int_0^{1-\gl}\om^{+1}_{+1}=\int_0^{1-\gl}\frac{2tdt}{1-t^2}=& \int_0^{1-\gl} \frac{dt}{1-t}- \int_0^{1-\gl}\frac{dt}{1+t}
=-\log \gl-\log 2+O(\gl),\\
\int_0^{1-\gl}\om^{-1}_{+1}=\int_0^{1-\gl}\frac{2dt}{1-t^2}=& \int_0^{1-\gl} \frac{dt}{1-t}+\int_0^{1-\gl}\frac{dt}{1+t}
=-\log \gl+\log 2+O(\gl).
\end{align*}
By the usual regularization process we see that $\evaM_\sha(\om^{+1}_{+1})=T-\log 2$ and $\evaM_\sha(\om^{-1}_{+1})=T+\log 2$.
The rest of the proof is essentially the same as that for the corresponding result of MZVs and is thus left
to the interested reader.
\end{proof}

\begin{re}
The following question was posed by the referee: is it possible to choose the images of $M(1)$ and $M(\wc{1})$
independently? This may be possible and may play some important roles when we consider generating functions
of AMMVs in general. However, the additional complication is unnecessary for our purpose in this paper.
\end{re}

The following structural result is useful in proving the comparison theorem between the two different ways
of regularization of AMMVs. Let $\AMMV_w$ be the $\Q$-vector space generated by all the AMMVs of weight $w$
and denote all its subspaces similarly.

\begin{thm} \label{thm:AMMV=CMZV4}
For any $w\in\N_0$, we have
\begin{align*}
\dim_{\Q[i]} \AMMV_w\otimes_\Q \Q[i]=\dim_{\Q[i]}  \CMZV^4_w\otimes_\Q \Q[i].
\end{align*}
\end{thm}
\begin{proof}
For any $N\in\N$, $\bfk=(k_1,\ldots,k_r)\in\N^r$, $\bfeps=(\eps_1, \dots, \eps_r)\in\{\pm 1\}^r$, and $\bfgs=(\sigma_1, \dots, \sigma_r)\in\{\pm 1\}^r$, the partial sum of AMMV truncated at $2N$ is
\begin{align}\label{equ:AMMZ2CMZV4}
M_\bfgs^\bfeps(\bfk)_N:={}&\sum_{2N\ge m_1>\cdots>m_r>0} \frac{(1+\eps_1(-1)^{m_1})\sigma_1^{(2m_1+1-\eps_1)/4} \cdots (1+\eps_r(-1)^{m_r})\sigma_r^{(2m_r+1-\eps_r)/4}}{m_1^{k_1} \cdots m_r^{k_r}} \notag\\
={}&\sum_{2N\ge m_1>\cdots>m_r>0} \frac{(1+\eps_1(-1)^{m_1})\mu_1^{m_1+(1-\eps_1)/2} \cdots (1+\eps_r(-1)^{m_r})\mu_r^{m_r+(1-\eps_r)/2}}{m_1^{k_1} \cdots m_r^{k_r}} \notag\\
={}&\bigg(\prod_{j=1}^r \mu_j^{(1-\eps_j)/2} \bigg)\sum_{2N\ge m_1>\cdots>m_r>0} \frac{(1+\eps_1(-1)^{m_1})\mu_1^{m_1} \cdots (1+\eps_r(-1)^{m_r})\mu_r^{m_r}}{m_1^{k_1} \cdots m_r^{k_r}}
\end{align}
where $\mu_j=i$ if $\sigma_j=-1$ and $\mu_j=1$ if $\sigma_j=1$. Taking $N\to\infty$ we see that $\AMMV_w\subseteq \CMZV^4_w\otimes_\Q \Q[i]$.

On the other hand, for any $\xi_j\in\{e^{\pi i/2 },e^{\pi i},e^{3\pi i/2},e^{2\pi i}\}=\{\pm1,\pm i\}$, we have
\begin{align*}
\xi_j^{n_j}:=
\left\{
  \begin{array}{ll}
\{\pm1,\pm i (-1)^{{(n_j-1)}/{2}}\}\ &\quad \hbox{if $n_j$\ \text{odd};} \\
\{1,\pm (-1)^{{n_j}/{2}}\}\ &\quad \hbox{if $n_j$\ \text{even},}
  \end{array}
\right.
\end{align*}
and therefore according to the definition of colored MZVs of level four,
\begin{align} 
\Li_{k_1,k_2,\ldots, k_r}(\xi_1,\dots,\xi_r)={}&\sum\limits_{n_1>\cdots>n_r>0}
\frac{\xi_1^{n_1}\dots \xi_r^{n_r}}{n_1^{k_1} \dots n_r^{k_r}}     \notag\\
={}&\sum_{\gd_1=0}^1 \cdots \sum_{\gd_r=0}^1
\sum_{\substack{n_1>\cdots>n_r>0 \\ n_j\equiv \gd_j \pmod{2} \ \forall j }}\frac{\xi_1^{n_1}
 \dots \xi_r^{n_r}}{n_1^{k_1} \dots n_r^{k_r}} \in \AMMV \otimes_\Q \Q[i]. \label{equ:CMZV2AMMV}
\end{align}
Hence, $\CMZV^4_w \subseteq \AMMV_w\otimes_\Q \Q[i]$. This concludes the proof of the theorem.
\end{proof}

We can now apply the same mechanism as in \cite{IKZ2006} to derive the following regularized DBSFs.

\begin{thm}\label{thm:RDSoverR}
Define an $\R$-linear map $\rho:\R[T]\to \R[T]$ by
$$\rho(e^{Tu})=\exp\left(\sum_{n=2}^\infty
\frac{(-1)^n}{n}\zeta(n)u^n\right)e^{(T-\log 2)u},\qquad |u|<1.$$
Then for any $\bfw\in\fA^1$ one has
\begin{equation*}
\evaML_\sha(\bfw;T)= \rho\circ\evaML_{\cstt}(\bfw;T).
\end{equation*}
\end{thm}
\begin{proof}
It is possible to modify the proof contained in \cite{IKZ2006} and adapt to our situation.
However, the process is tedious and requires a lot computation and estimates. See \cite{LiZhao2024}
for a detailed presentation for the so-called $\mu$-multiple Hurwitz zeta values. By taking $\mu=2$
and $m$'s to be arbitrary positive integers in \cite{LiZhao2024} we recover the proof of the comparison theorem
for the regularization of the non-alternating MMVs.

Instead of reinventing the wheel for AMMVs, we can use Theorem~\ref{thm:AMMV=CMZV4} and the regularization
of CMZVs to prove the current theorem. Indeed, using the same notation as in Theorem~\ref{thm:AMMV=CMZV4},
by \eqref{equ:AMMZ2CMZV4}, we see that
\begin{align}\label{equ:AMMZ2CMZV4a}
M_\bfgs^\bfeps(\bfk)_N={}&\bigg(\prod_{j=1}^r \mu_j^{(1-\eps_j)/2} \bigg)
\sum_{\gd_1=\pm 1}\cdots \sum_{\gd_r=\pm1}
\Big( \tilde{\eps}_1(\gd_1)\cdots \tilde{\eps}_r(\gd_r)\Big) \Li^{(2N)}_\bfk(\gd_1\mu_1,\dots,\gd_r\mu_r)
\end{align}
where $\tilde{\eps}_j(1)=1$ and $\tilde{\eps}_j(-1)=\eps_j$, $\mu_j=i$ or 1 according to $\mu_j^2=\sigma_j$, and $\Li^{(2N)}$ denotes the partial sum of
the CMZV truncated at $2N$. By the regularization scheme for CMVZs (see \cite[Ch.~13]{Zhao2016}), as $N\to \infty$,
\begin{equation*}
 \Li^{(2N)}_{\bfk}(\gd_1\mu_1,\dots,\gd_r\mu_r)=
\Li_\bfk^\ast(\gd_1\mu_1,\dots,\gd_r\mu_r;T+\log 2)\big|_{T=\log N+\gamma}+O(N^{-1}\log^r N),
\end{equation*}
where $\Li_\bfk^\ast(\gd_1\mu_1,\dots,\gd_r\mu_r;T)\in \CC[T]$. Therefore, setting $\bfw=\om_0^{k_1-1}\om_{\sigma_1}^{\eps_1}\cdots \om_0^{k_r-1}\om_{\sigma_r}^{\eps_r}$, we get
\begin{equation}\label{equ:stuffleM}
\evaM_{\cstt}(\bfw;T)=
\bigg(\prod_{j=1}^r \mu_j^{(1-\eps_j)/2} \bigg)
\sum_{\gd_1=\pm 1}\cdots \sum_{\gd_r=\pm1}
\Big( \tilde{\eps}_1(\gd_1)\cdots \tilde{\eps}_r(\gd_r)\Big)
\Li_\bfk^\ast(\gd_1\mu_1,\dots,\gd_r\mu_r;T+\log2).
\end{equation}
Moreover, if we define the $\CC$-linear map $\tau:\CC[T]\to\CC[T]$
such that $\tau(T)=T-\log2$ and define the  $\CC$-linear map $\tilde{\rho}:\CC[T]\to\CC[T]$ such that
$$
\tilde{\rho}(e^{Tu})=\exp\left(\sum_{n=2}^\infty\frac{(-1)^n}{n}\zeta(n)u^n\right)e^{Tu},\qquad |u|<1,
$$
then for $\rho=\tilde{\rho}\circ\tau$ we get
\begin{equation}\label{equ:rhoMap}
\rho(\Li_\bfk^\ast(\gd_1\mu_1,\dots,\gd_r\mu_r;T))=\Li_\bfk^\sha(\gd_1'\mu_1',\dots,\gd_r'\mu_r';T)
\end{equation}
where $\gd_j'=\prod_{n=1}^j \gd_n$, $\mu_j'=\prod_{n=1}^j \mu_n$ and $\Li_\bfk^\sha(\bfxi;T)\in \CC[T]$ such that
\begin{equation*}
\int_0^{1-\gl} \om_0^{k_1-1}\frac{dt}{\xi_1^{-1}-t} \cdots \om_0^{k_1-1}\frac{dt}{\xi_r^{-1}-t}
=\Li_\bfk^\sha(\bfxi;-\log \gl) +O(\gl \log^r \gl)
\end{equation*}
as $\gl\to 0^+$, for all $\bfk\in\N^r$ and $\bfxi\in\{\pm1,\pm i\}^r$.
It is clear that $\rho=\tilde{\rho}\circ\tau$ satisfies the conditions given in the statement of the theorem.

On the other hand, setting $\gk_j=\prod_{n=1}^j \sigma_n$ for $j=1,\dots,r$,
we obtain from \eqref{AMMV-iterated-integrals} the shuffle regularized value of
$M_\bfgs^{\bfeps}(\bfk)$
\begin{align}\label{equ:shaRegIntegral}
  \int_0^{1-\gl}  \om_0^{k_1-1}W_{\gk_1}^{\eps_1,\eps_2}
    \cdots \om_0^{k_{r-1}-1}W_{\gk_{r-1}}^{\eps_{r-1},\eps_r}  \om_0^{k_r-1}\om_{\gk_r}^{\eps_r}.
\end{align}
To save space, for any 4-th root of unity $\mu$ we put
\begin{equation*}
  \tx(\mu)=\frac{dt}{\mu^{-1}-t}.
\end{equation*}
Then by definition \eqref{equ:1-forms}
\begin{equation*}
\aligned
 \om_{+1}^{-1}:={}& \frac{2dt}{1-t^2}=\tx(1)-\tx(-1), {}&
\quad \om_{-1}^{-1}:={}&\frac{-2dt}{1+t^2}=i\big(\tx(i)-\tx(-i)\big), \\
 \om_{+1}^{+1}:={}&\frac{2tdt}{1-t^2}=\tx(1)+\tx(-1), {}&
\quad \om_{-1}^{+1}:={}&\frac{-2tdt}{1+t^2}=\tx(i)+\tx(-i).
\endaligned
\end{equation*}
Let $\eta_j=\eps_j\eps_{j+1}$ for all $j=1,\dots, r$, where $\eps_{r+1}=1$.
We see that for all $j=1,\dots,r-1$ we always have
\begin{equation*}
  W_{\gk_j}^{\eps_j,\eps_{j+1}}=
(\mu_j')^{(\eps_{j+1}-\eps_j)/2}  \big(\tx(\mu_j')a+\eta_j\tx(-\mu_j')\big)
\end{equation*}
where $\mu_j'=\prod_{n=1}^j \mu_n$ for all $j=1,\dots,r$.
Further, the last 1-form in \eqref{equ:shaRegIntegral} can be written as
\begin{equation*}
\om_{\gk_r}^{\eps_r} =
(\mu'_r)^{(1-\eps_r)/2} \big(\tx(\mu_r')a+\eta_r\tx(-\mu_r')\big).
\end{equation*}
Hence the shuffle regularized value of $M_\bfgs^{\bfeps}(\bfk)$ is (setting $\eta_{r+1}=1$)
\begin{align*}
& \int_0^{1-\gl}  \om_0^{k_1-1}W_{\gk_1}^{\eps_1,\eps_2}
    \cdots \om_0^{k_{r-1}-1}W_{\gk_{r-1}}^{\eps_{r-1},\eps_r}  \om_0^{k_r-1}\om_{\gk_r}^{\eps_r} \\
={}&  \bigg(\prod_{j=1}^r (\mu'_j)^{(1-\eta_j)/2-(1-\eta_{j+1})/2}\bigg)
\int_0^{1-\gl}  \om_0^{k_1-1}  \big(\tx(\mu_1')a+\eta_1\tx(-\mu_1')\big)
\cdots  \om_0^{k_r-1}  \big(\tx(\mu_r')a+\eta_r\tx(-\mu_r')\big) \\
={}& \bigg(\prod_{j=1}^r (\mu'_j/\mu'_{j-1})^{(1-\eta_j)/2}\bigg) \sum_{\gd_1'=\pm 1}\cdots \sum_{\gd_r'=\pm1}
 \Big(\tilde{\eta}_1(\gd_1')\cdots \tilde{\eta}_r(\gd_r') \Big)\Li_\bfk^\sha(\gd_1'\mu'_1,\dots,\gd_r'\mu'_r ;-\log \gl) +O(\gl \log^r \gl)
\end{align*}
where $\mu'_0=1$, $\tilde{\eta}_j(1)=1$ and $\tilde{\eta}_j(-1)=\eta_j$ for all $j=1,\dots,r$.
Moreover, it is easy to check that for all $j=1,\dots,r$
\begin{equation}\label{equ:epsgdRel}
 \tilde{\eps}_j(\gd_1\gd_2)=\tilde{\eps}_j(\gd_1)\tilde{\eps}_j(\gd_2),
\quad\text{and}\quad \widetilde{\eps_j\eps_{j+1}}(\gd)=\tilde{\eps}_j(\gd)\tilde{\eps}_{j+1}(\gd) \quad\forall \gd,\gd_1,\gd_2=\pm1.
\end{equation}
Thus by the change of indices $\gd_j'=\prod_{n=1}^j \gd_n$ for all $j=1,\dots,r$ we have
\begin{align*}
\tilde{\eta}_1(\gd_1')\cdots \tilde{\eta}_r(\gd_r')={}&
\widetilde{\eps_1\eps_2}(\gd_1)\widetilde{\eps_2\eps_3}(\gd_1\gd_2)\cdots \widetilde{\eps_{r-1}\eps_r}(\gd_1\gd_2\cdots \gd_{r-1} )\tilde{\eps}_r(\gd_1\gd_2\cdots \gd_r)
=\tilde{\eps}_1(\gd_1) \cdots \tilde{\eps}_r(\gd_r).
\end{align*}
Hence
\begin{align*}
& \int_0^{1-\gl}  \om_0^{k_1-1}W_{\gk_1}^{\eps_1,\eps_2}
    \cdots \om_0^{k_{r-1}-1}W_{\gk_{r-1}}^{\eps_{r-1},\eps_r}  \om_0^{k_r-1}\om_{\gk_r}^{\eps_r} \\
={}& \bigg(\prod_{j=1}^r \mu_j^{(1-\eta_j)/2}\bigg) \sum_{\gd_1=\pm 1}\cdots \sum_{\gd_r=\pm1}
\Big(\tilde{\eps}_1(\gd_1) \cdots \tilde{\eps}_r(\gd_r)\Big)
\Li_\bfk^\sha(\gd_1'\mu'_1,\dots,\gd_r'\mu'_r ;-\log \gl) +O(\gl \log^r \gl)
\end{align*}
where we recall that $\gd_j'=\prod_{n=1}^j \gd_n$ and $\mu'_j=\prod_{n=1}^j \mu_n$. This implies that
\begin{equation}\label{equ:shuffleM}
\evaM_{\sha}(\bfw;T)=
\bigg(\prod_{j=1}^r \mu_j^{(1-\eta_j)/2}\bigg) \sum_{\gd_1=\pm 1}\cdots \sum_{\gd_r=\pm1}
\Big(\tilde{\eps}_1(\gd_1) \cdots \tilde{\eps}_r(\gd_r)\Big)\Li_\bfk^\sha(\gd_1'\mu'_1,\dots,\gd_r'\mu'_r ;T).
\end{equation}
The theorem now follows from \eqref{equ:stuffleM}, \eqref{equ:rhoMap} and \eqref{equ:shuffleM}.
As a final note, since all AMMVs are real numbers, we see that the theorem holds over $\R$.
\end{proof}

\begin{re} We know there should be $\Q$-linear relations among Euler sums or AMZVs that are not consequences
of the regularized DBSF, see \cite[Remark 3.5]{Zhao2010a}. Further, there are some octahedral relations
discovered in \cite{Zhao2008d} that are not consequences of regularized DBSF. Therefore, the same should hold for AMMVs by Theorem~\ref{thm:AMMV=CMZV4}.
\end{re}

Next, let's consider a weight three example using a regularized AMMV as defined by \cite[Definition 3.2]{YuanZh2014b}.

\begin{exa}  As a simple example of the above theorem, we first have
\begin{align*}
\evaM_\sha(\wc{1},\bar2)={}&\, \evaM_\sha\big(\bfp(\om_1^{-1}\om_0\om^1_{-1})\big)=\evaM_\sha(\om_1^{-1}\om_0\om^1_{-1}) \\
={}&\,  \evaM_\sha(\om_1^{-1}\sha \om_0\om^1_{-1}-\om_0\om_1^{-1}\om^1_{-1}-\om_0\om^1_{-1}\om_1^{-1})\\
={}&\,  \evaM_\sha(\om_1^{-1})M(\wc{2})-M(\bfq(\om_0\om_1^{-1}\om^1_{-1}))-M(\bfq(\om_0\om^1_{-1}\om_1^{-1})))\\
={}&\, (T+\log 2)M(\bar{2})-M(\wc{2},\bar1)-M(\wc{\bar2},\wc{\bar1}).
\end{align*}
On the other hand,
\begin{align*}
\evaM_{\cstt}( \wc{1},\bar2)
={}&\,\evaM_{\cstt}\big( \bfp(\tz_{1,1}^{-1}\tz_{2,-1}^1) \big)
=\evaM_{\cstt}\big( \bfp(\tz_{1,1}^{-1} \ast\tz_{2,-1}^1 -\tz_{2,-1}^1 \tz_{1,1}^{-1}) \big)\\
={}&\, \evaM_{\cstt}\big(\tz_{1,1}^{-1} \cst\tz_{2,-1}^1 - \bfp(\tz_{2,-1}^1 \tz_{1,1}^{-1})\big)
=(T+2\log 2)M(\bar{2})-M(\bar2,\wc{1}),
\end{align*}
since no stuffing can appear. By Theorem~\ref{thm:RDSoverR}, we get
\begin{align*}
(T+\log 2)M(\bar{2})-M(\wc{2},\bar1)-M(\wc{\bar2},\wc{\bar1})
&\,=\rho\Big((T+2\log 2)M(\bar{2})-M(\bar2,\wc{1})\Big) \\
&\,=(T+\log 2)M(\bar{2})-M(\bar2,\wc{1})\\
&\, \Lra  M(\wc{2},\bar1)+M(\wc{\bar2},\wc{\bar1})=M(\bar2,\wc{1}).
\end{align*}
We can check this identity numerically and find that both sides are   $\approx -0.7739912$.

For more complicated examples, it is in fact easier to use \eqref{equ:stuffleM} and \eqref{equ:shuffleM}
to express the regularized AMMVs in terms of regularized CMZVs.
\end{exa}

\begin{thm} \label{thm:dualMMVo} \emph{(Generalized Duality Relation)}
Let $\bfk=(k_1,\ldots,k_r),\bfl=(l_1,\ldots,l_r)\in\N^r$ and $\bfgs,\bfeps\in\{\pm 1\}^r$ with $\eps_1=-1$.
Then
\begin{equation}\label{equ:dualMMVo}
M \big( \om_0^{k_r} (\om_{\sigma_r}^{\eps_r})^{l_r} \cdots(\om_{\sigma_2}^{\eps_2})^{l_2}\om_0^{k_1}(\om_{\sigma_1}^{-1})^{l_1}  \big)
=M \big( (u_{\sigma_1}^{-1})^{l_1} u_0^{k_1} (u_{\sigma_2}^{\eps_2})^{l_2} \cdots (u_{\sigma_r}^{\eps_r})^{l_r} u_0^{k_r} \big),
\end{equation}
where $u_0=\om_{+1}^{-1}$, $u_{+1}^{-1}=\om_0$, $u_{-1}^{-1}=\om_{-1}^{-1}$, $u_{+1}^{+1}=\om_0+\om_{+1}^{+1}-\om_{+1}^{-1}$
and $u_{-1}^{+1}=\om_{+1}^{+1}-\om_{+1}^{-1}-\om_{-1}^{+1}$.
\end{thm}
\begin{proof}
Under the substitution $t\to \frac{1-t}{1+t}$, we get
\begin{align*}
dt\to \frac{-2dt}{(1+t)^2},
\quad 1-t^2 \to  \frac{4t}{(1+t)^2},
\quad  1+t^2\to  \frac{2(1+t^2)}{(1+t)^2} .
\end{align*}
Thus
\begin{align*}
\om_0=\frac{dt}{t}\to -\frac{2dt}{1-t^2}=-\om_{+1}^{-1},
\quad \om_{+1}^{-1}=\frac{2dt}{1-t^2}\to -\om_0,
\quad \om_{-1}^{-1}=\frac{-2dt}{1+t^2}\to -\om_{-1}^{-1}, \\
\quad \om_{+1}^{+1}=\frac{2tdt}{1-t^2}\to \frac{-(1-t)dt}{t(1+t)}=\om_{+1}^{-1}-\om_0-\om_{+1}^{+1}, \\
\quad \om_{-1}^{+1}=\frac{-2tdt}{1+t^2}\to \frac{2(1-t)dt}{(1+t^2)(1+t)}=\om_{+1}^{-1}-\om_{+1}^{+1}+\om_{-1}^{+1}.
\end{align*}
The theorem follows immediately.
\end{proof}

\begin{re}\label{rem:genDual}
(1). Note that on the left-hand side of the generalized duality relation \eqref{equ:dualMMVo} the last 1-form
must be $\om_{\pm 1}^{-1}$. It is not difficult to adopt a regularization procedure to derive the generalized
duality relation with ending 1-form equal to $\om_{\pm 1}^{+1}$. We will leave this to the interested reader.

(2). If $\om_{\pm1}^{+1}$ appears on the left-hand side of \eqref{equ:dualMMVo} then there are always more than
one term on the right since the terms produced on the right-hand side are all distinct and cannot cancel each other.
\end{re}

\begin{exa}
To illustrate Remark \ref{rem:genDual} concretely, we have found the following
generalized duality relations first of which does not involve $\om_{\pm1}^{+1}$:
\begin{align*}
&\begin{aligned}
M(\wc2,1,\cbar2)={}&\int_0^1 \om_0\om^{-1}_{+1}\om^{-1}_{+1}\om_0\om^{-1}_{-1}
=\int_0^1 u_{-1}^{-1}u_0 u_{+1}^{-1}u_{+1}^{-1} u_0
=M(\cbar1,\bar1,\cbar3)\approx -1.514104292,
\end{aligned}\\
&\begin{aligned}
M(2,\wc{1},\cbar2)={}&\int_0^1 \om_0\om_{+1}^{-1}\om_{+1}^{+1}\om_0\om_{-1}^{-1}=\int_0^1 u_{-1}^{-1}u_0u_{+1}^{+1}u_{+1}^{-1}u_0 \\
={}& M(\cbar{1},\bar1,\wc{3})+M(\cbar{1},\bar1,\wc{1},\wc{2})
+M(\bar1,\cbar{1},1,\wc{2})\approx  -0.872885216,
\end{aligned}\\
&\begin{aligned}
M(\cbar2,\wc{1},\cbar1)={}&\int_0^1 \om_0\om^{+1}_{-1}\om^{+1}_{-1}\om^{-1}_{+1}=\int_0^1 u_{+1}^{-1}u_{-1}^{+1}u_{-1}^{+1}u_0\\
={}&M(\wc{2},\wc{1},\wc{1})+M(\wc{2},1,\wc{1})+M(2,\cbar1,\cbar1)+M(\bar2,\bar1,\wc{1})+M(\cbar2,\wc{1},\cbar1)\\
&\quad-M(2,1,\wc{1})-M(\wc{2},\cbar1,\cbar1)
-M(2,\wc{1},\wc{1})-M(\cbar2,\cbar1,\wc{1})\approx 0.0642145399,
\end{aligned}\\
&\begin{aligned}
M(\bar2,\bar3,\cbar{1})={}&\int_0^1 \om_0\om_{-1}^{+1}\om_0^2\om_{+1}^{-1}\om_{-1}^{-1}=\int_0^1 u_{-1}^{-1}u_{+1}^{-1}u_0^2u_{-1}^{+1}u_0\\
={}&-M(\bar1,\cbar{2},1,\wc{1},\wc{1})-M(\cbar{1},\bar2,\wc{1},1,\wc{1})
+M(\bar1,\cbar{2},1,\cbar1,\cbar1)\approx 0.0463195291,
\end{aligned}\\
&\begin{aligned}
M(\cbar3,\wc{1},\bar2,\wc{1})={}&\int_0^1 \om_0^2\om_{-1}^{+1}\om_{-1}^{-1}\om_0\left(\om_{+1}^{-1}\right)^2=\int_0^1 \left(u_{+1}^{-1}\right)^2u_0u_{-1}^{-1}u_{-1}^{+1}u_0^2\\
={}&M(3,\cbar1,\bar1,1,\wc{1})+M(\wc{3},\bar1,\cbar1,1,\wc{1})
-M(3,\cbar1,1,\bar1,\wc{1})\approx 0.0072347087.
\end{aligned}
\end{align*}
\end{exa}

We end this section by the following parity theorem of AMMVs of arbitrary depth which confirms one of
our previous conjectures (see \cite[Conj. 5.2]{XuZhao2020b}). The key step is the 
general parity result of Panzer on colored MZVs (see \cite{Panzer2017}). For $\sigma,\eps=\pm1$, we define
\begin{equation}\label{equ:iot_gs_eps}
    \iota(\sigma,\eps)=(-1)^{(1-\sigma)(1-\eps)/4}=
\left\{
  \begin{array}{ll}
    -1, \quad \ & \hbox{if $\sigma=\eps=-1$;} \\
    1, \quad & \hbox{otherwise.}
  \end{array}
\right.
\end{equation}
For $d,w\in\N$ we set $\CMZV_w^d$ (resp.\ $\AMMV_w^d$) to be the $\Q$-span of all CMZVs 
(resp.\ AMMVs) of weight $w$ and depth $d$. For $w\ge d\ge 2$, we define
\begin{align*}
& \CMZV_w^{<d}=\sum_{0<\ell<d}\CMZV_w^{\ell}+\sum_{\substack{w_1+w_2=w,\ w_1,w_2>0\\ d_1+d_2=d,\ d_1,d_2>0} } \CMZV_{w_1}^{d_1}\cdot \CMZV_{w_2}^{d_2},\\
& \AMMV_w^{<d}=\sum_{0<\ell<d}\AMMV_w^{\ell} +\sum_{\substack{w_1+w_2=w,\ w_1,w_2>0\\ d_1+d_2=d,\ d_1,d_2>0} } \AMMV_{w_1}^{d_1}\cdot\AMMV_{w_2}^{d_2}.
\end{align*}

\begin{thm}\label{thm:MMVparity}
Suppose $\bfk$ has depth $r\ge 2$ with weight $w$. For all admissible $M_\bfgs^\bfeps(\bfk)$
\begin{equation*}
\Big(1-(-1)^{r+w} \prod_{j=1}^r \iota(\sigma_j,\eps_j)\Big) M_\bfgs^\bfeps(\bfk)\in \AMMV_w^{<r}.
\end{equation*} 
\end{thm}

\begin{proof}
Applying Panzer's parity theorem  (see \cite[Theorem~1.3]{Panzer2017}) to
the right-hand side of \eqref{equ:AMMZ2CMZV4a} (after taking $N\to\infty$), we get
\begin{align*}
M_\bfgs^\bfeps(\bfk) \equiv {}& (-1)^{r+w}\bigg(\prod_{j=1}^r \mu_j^{(1-\eps_j)/2} \bigg)
\sum_{\gd_1=\pm 1}\cdots \sum_{\gd_r=\pm1}
\Big( \tilde{\eps}_1(\gd_1)\cdots \tilde{\eps}_r(\gd_r)\Big) \Li_\bfk(\gd_1\mu_1^{-1},\dots,\gd_r\mu_r^{-1})
\end{align*}
modulo $\CMZV_w^{<r}$. But it is clear that  $\mu_j^{-1}=\mu_j\sigma_j$ for all $j=1,\dots,r$. By the change of variables
$\gd_j\to \gd_j \sigma_j$ we see that
\begin{align*}
M_\bfgs^\bfeps(\bfk) \equiv {}& (-1)^{r+w}\bigg(\prod_{j=1}^r \mu_j^{(1-\eps_j)/2} \bigg)
\sum_{\gd_1=\pm 1}\cdots \sum_{\gd_r=\pm1}
\Big( \tilde{\eps}_1(\gd_1\sigma_1)\cdots \tilde{\eps}_r(\gd_r\sigma_r)\Big) \Li_\bfk(\gd_1\mu_1,\dots,\gd_r\mu_r).
\end{align*}
From \eqref{equ:epsgdRel} and the fact that $\tilde{\eps}_j(\sigma_j)=\iota(\sigma_j,\eps_j)$ we see that
\begin{equation*}
\Big(1-(-1)^{r+w} \prod_{j=1}^r \iota(\sigma_j,\eps_j)\Big) M_\bfgs^\bfeps(\bfk)\in \RRe \CMZV_w^{<r}.
\end{equation*}
By \eqref{equ:CMZV2AMMV} we know that for all weight $k$ and depth $d$, 
\begin{equation*}
\CMZV_k^{d}\subseteq \AMMV_k^{d}+i\cdot \AMMV_k^{d}.
\end{equation*}
Therefore 
\begin{equation*}
\RRe  \Big(\CMZV_{w_1}^{d_1}\cdot\CMZV_{w_2}^{d_2}\Big) \subseteq \AMMV_{w_1}^{d_1}\cdot\AMMV_{w_2}^{d_2}.
\end{equation*}
The theorem now follows immediately.
\end{proof}
 
We would like to point out that the proof of Theorem~\ref{thm:MMVparity} does not provide the reduction formula 
when $(-1)^{r+w} \prod\limits_{j=1}^r \iota(\sigma_j,\eps_j)=-1$. However, by residue computation as used in \cite{XuWang2020} it is possible
to find this explicitly, at least when $r=3$, even though it is much more complicated than the proof above. 
In fact, we obtained all the following identities by this method.

\begin{exa}
We have the following reduction identities:
\begin{align*}
M(\bar1,\wc2,\wc1)={}&M(\wc3,\cbar1)-2M(\wc3,\wc1)-M(\wc2,\wc2)-2M(\wc2,\cbar1)\log2-\tfrac{\pi}{2}M(2,\cbar1)
		+\tfrac{35}{4}\zeta(3)\log2-\tfrac{\pi^2}{2}\log^2 2-\tfrac{\pi^4}{16},\\
M(2,\wc2,\cbar1)={}&M(\wc2,\cbar3)+M(\cbar3,2)+2M(\wc3,\cbar2)+3M(\wc4,\cbar1)+\tfrac{\pi^2}{8}M(\wc2,\cbar1)+\tfrac{\pi}{2}M(2,\wc2)+M(\cbar2,2)\log2\\
		 &+7G\zeta(3)+\tfrac{7\pi}{2}\zeta(3)\log2-6\beta(4)\log2-\tfrac{\pi^2}{12}G\log2-\tfrac{\pi^5}{24},\\
M(\bar2,\cbar3,\wc2)={}&3M(\cbar3,\wc4)+6M(\cbar4,\wc3)+6M(\cbar5,\wc2)+M(\cbar5,\bar2)-\tfrac{\pi^2}{8}M(\cbar3,\wc2)+\tfrac{\pi^2}{4}M(\bar3,\cbar2)\\
		&+21\zeta(3)\beta(4)+\tfrac{5\pi^2}{32}G\zeta(3)-\tfrac{173\pi^7}{18432},\\ 
M(\wc2,\wc4,\bar2)={}&3M(\wc4,\wc4)+8M(\wc5,\wc3)+10M(\wc6,\wc2)-M(\wc6,\cbar2) +\tfrac{\pi^2}{4}M(\wc4,\cbar2)+\tfrac{\pi^2}{24}M(\wc4,\wc2)\\
		&+\pi M(4,\cbar3)+2\pi M(5,\cbar2)-\tfrac{217}{8}\zeta(3)\zeta(5)-\tfrac{3\pi}{2}\zeta(3)\beta(4)-\tfrac{\pi^3}{16}G\zeta(3)+\tfrac{13\pi^8}{2880},\\
M(\bar1,\bar2,\wc1)={}&-2M(\bar2,\bar1)\log2-M(\bar2,\wc2)-2M(\bar3,\wc1)-M(3,\wc1)
		-\tfrac{27}{8}\zeta(3)\log2+\tfrac{\pi^2}{4}\log^2 2,\\
M(2,\wc2,\bar2)={}&-3M(2,\wc4)-3M(\bar2,\wc4)-4M(3,\wc3)-4M(\bar3,\wc3)-3M(4,\wc2)-3M(\bar4,\wc2)-\tfrac{\pi^2}{24}M(2,\wc2)\\
		&-\tfrac{\pi^2}{8}M(\bar2,\wc2)-\pi M(\wc2,\cbar3)-\pi M(\cbar2,\wc3)
		-\pi M(\wc3,\cbar2)-\pi M(\cbar3,\wc2)-7\pi G\zeta(3)+\tfrac{5\pi^6}{192},\\		 M(2,\cbar2,3)={}&4M(2,\cbar5)+3M(3,\cbar4)-2M(\cbar2,5)-4M(5,\cbar2)+\tfrac{1}{4}\zeta(3)M(2,\cbar2)+\tfrac{\pi^2}{12}M(2,\cbar3)+2\pi M(\wc3,\wc3)
		\\
		 &+\tfrac{\pi^2}{12}M(\cbar2,3)+\tfrac{\pi^2}{24}M(3,\cbar2)+\tfrac{3\pi}{2}M(\wc2,\wc4)+\tfrac{3\pi}{2}M(\wc4,\wc2)-\tfrac{1}{4}G\zeta(5)
        +\tfrac{\pi^2}{24}G\zeta(3)-\tfrac{49\pi}{8}\zeta^2(3)+\tfrac{\pi^7}{96},\\		 M(\bar3,\wc2,\bar3)={}&8M(\bar3,\wc5)-2M(\wc2,6)-\tfrac{3}{16}\zeta(3)M(\wc2,\cbar3)-\tfrac{3}{8}\zeta(3)M(\bar3,\wc2)
        -\tfrac{\pi^2}{3}M(\bar3,\wc3)-\tfrac{\pi^2}{2}M(\bar3,\bar3)\\
		 &+\tfrac{\pi^2}{2}M(\cbar3,\cbar3)-\tfrac{\pi^2}{2}M(\bar4,\wc2)+18M(\bar4,\wc4)+24M(\bar5,\wc3)+20M(\bar6,\wc2)-2M(6,\wc2)+\tfrac{\pi^8}{960}.
	\end{align*}
\end{exa}

\section{Some results on AMSVs and AMTVs}\label{AMTVs-AMSVs}
In this section, we study alternating multiple $S$- and $T$-values with arguments of the form $(\bar2,1,\ldots,1,\bar1)$, 
with or without the bars at the first and last components. We will establish some explicit relations between alternating 
multiple $S$- and $T$-values by using the integrals of arctangent function.

\begin{lem}\label{lem-orr-2017} \emph{(cf. \cite[Eqs. (2.2) and (2.3)]{Orr2017})}
For $p\in \N$, we have
\begin{align*}
\int_0^{\pi/2} x^p \cot(x)dx={}&\left(\frac{\pi}{2}\right)^p\left\{ \log2+\sum_{k=1}^{[p/2]} \frac{p!(-1)^k(4^k-1)}{(p-2k)!(2\pi)^{2k}}\zeta(2k+1)\right\}
    +\delta_{[p/2],p/2}\frac{p!(-1)^{p/2}}{2^p}\zeta(p+1),\\
\int_0^{\pi/4} x^p \cot(x)dx={}&\frac1{2}\left(\frac{\pi}{4}\right)^p\left\{ \log2+\sum_{k=1}^{[p/2]} \frac{p!(-1)^k(4^k-1)}{(p-2k)!(2\pi)^{2k}}\zeta(2k+1)\atop -\sum_{k=1}^{[(p+1)/2]} \frac{p!(-4)^k\beta(2k)}{(p+1-2k)!\pi^{2k-1}}\right\}
    +\delta_{[p/2],p/2}\frac{p!(-1)^{p/2}}{2^p}\zeta(p+1),
\end{align*}
where $\delta$ is the Kronecker symbol, and the Dirichlet beta function $\beta(s)$ is defined by
\begin{align}\label{defn-Dbeta-function}
\beta(s):=\sum_{n=1}^\infty \frac{(-1)^{n-1}}{(2n-1)^s}\quad (\Re(s)>0).
\end{align}
When $s=2$, $\beta(2)=G$ is known as Catalan's constant.
\end{lem}

\begin{pro}\label{pro-int-arc}
For $p\in \N$, we have
\begin{align*}
\int_0^1 \arctan^p(x)dx
={}&\left(\frac{\pi}{4}\right)^p+p\left(\frac1{2^p}-\delta_{1,p}\right)\Big(\frac{\pi}{2}\Big)^{p-1}\log2\\
&-p\left(\frac{\pi}{2}\right)^{p-1}\sum_{k=1}^{p-1}(-1)^k\binom{p-1}{k}\left\{\left(1-\frac1{2^{k+1}}\right)\sum_{j=1}^{[k/2]} \frac{k!(-1)^j(4^j-1)}{(k-2j)!(2\pi)^{2j}}\zeta(2j+1)\atop +\frac1{2^{k+1}}\sum_{j=1}^{[(k+1)/2]} \frac{k!(-4)^j\beta(2j)}{(k+1-2j)!\pi^{2j-1}}\right\}\\
\in{}& \Q[\log2,\pi,\beta(2),\zeta(3),\beta(4),\zeta(5),\cdots].
\end{align*}
\end{pro}
\begin{proof}
Letting $x=\tan(t)$ and using integration by parts, we note that the integral on the left-hand side can be rewritten as
\begin{align}\label{arctan-cot}
&\int_0^1 \arctan^p(x)dx=\int_0^{{\pi}/{4}} t^p d\tan(t)=\left(\frac{\pi}{4}\right)^p-p\int_0^{\pi/4} t^{p-1} \tan(t)dt\nonumber\\
={}&\left(\frac{\pi}{4}\right)^p-p\int_{\pi/4}^{\pi/2} \left(\frac{\pi}{2}-u\right)^{p-1} \cot(u)du\nonumber\\
={}&\left(\frac{\pi}{4}\right)^p-p\frac{\pi^{p-1}}{2^p} \log2-p\sum_{k=1}^{p-1}(-1)^k\binom{p-1}{k} \left(\frac{\pi}{2}\right)^{p-1-k} \int_{\pi/4}^{\pi/2} u^k \cot(u)du.
\end{align}
Thus, applying Lemma \ref{lem-orr-2017} yields the desired description.
\end{proof}

\begin{exa}\label{eg:ArcTanPowerIntegral}
As examples, here are a few evaluations of integrals of arc tangent powers over $[0,1]$:
\begin{align*}
&\int_0^1 \arctan(x)dx=\frac{\pi}{4}-\frac1{2}\log2,\\
&\int_0^1 \arctan^2(x)dx=\frac{\pi^2}{16}+\frac1{4}\pi\log2-G,\\
&\int_0^1 \arctan^3(x)dx=\frac{\pi^3}{64}+\frac{3}{32}\pi^2\log2-\frac3{4}\pi G+\frac{63}{64}\zeta(3),\\
&\int_0^1 \arctan^4(x)dx=\frac{\pi^4}{256}+\frac{\pi^3}{32}\log2-\frac3{8}\pi^2 G-\frac{9}{64}\pi\zeta(3)+3\beta(4).
\end{align*}
\end{exa}

\begin{re}
In \cite[p. 122]{Adams-Hippsley}, one can find the Maclaurin series expansion
\begin{align}\label{arctan^p-expansion}
\arctan^{p}(x)=p!\sum_{k_0=1}^{\infty}(-1)^{k_0-1}\frac{x^{2k_0+p-2}}{2k_0+p-2}\prod_{\alpha=1}^{p-1}\left(\sum_{k_{\alpha}=1}^{k_{\alpha-1}}\frac{1}{2k_{\alpha}+p-\alpha-2}\right)
\end{align}
for $p\in\mathbb{N}$. Integrating \eqref{arctan^p-expansion} from 0 to 1 results in
\begin{align*}
\int_0^1 \arctan^{p}(x)dx
={}&p!\sum_{k_0=1}^{\infty}\frac{(-1)^{k_0-1}}{(2k_0+p-1)(2k_0+p-2)}
\prod_{\alpha=1}^{p-1}\left(\sum_{k_{\alpha}=1}^{k_{\alpha-1}}\frac{1}{2k_{\alpha}+p-\alpha-2}\right)\\
={}&p!\sum_{k_0=1}^{\infty}\frac{(-1)^{k_0-1}}{2k_0+p-2}
\prod_{\alpha=1}^{p-1}\left(\sum_{k_{\alpha}=1}^{k_{\alpha-1}}\frac{1}{2k_{\alpha}+p-\alpha-2}\right)\\
&-p!\sum_{k_0=1}^{\infty}\frac{(-1)^{k_0-1}}{2k_0+p-1}
\prod_{\alpha=1}^{p-1}\left(\sum_{k_{\alpha}=1}^{k_{\alpha-1}}\frac{1}{2k_{\alpha}+p-\alpha-2}\right).
\end{align*}
Therefore, we have
\begin{align*}
\int_0^1 \arctan^{p}(x)dx=
\begin{cases}
\frac{(-1)^{\tfrac{p}{2}}p!}{2^p}I_p  &\text{if $p$ is even;}\\
\frac{(-1)^{\tfrac{p+1}{2}}p!}{2^p}J_p  &\text{if $p$ is odd,}
\end{cases}
\end{align*}
where (setting $q=\lfloor{(p-1)/2}\rfloor$)
\begin{align*}
I_p=M(\bar1,\wc1,\{1,\wc1\}_q)+M(\cbar1,\wc1,\{1,\wc1\}_q),\quad
J_p=M(\cbar1,\{1,\wc1\}_q)-M(\bar1,\{1,\wc1\}_q),
\end{align*}
where $\{S\}_m$ means the string $S$ repeats itself $m$ times. For example,
\begin{align*}
&\int_0^1 \arctan(x)dx=\frac{1}{2}\big(M(\bar1)-M(\cbar1)\big),\\
&\int_0^1 \arctan^2(x)dx=-\frac{1}{2}\big(M(\bar1,\wc1)+M(\cbar1,\wc1)\big),\\
&\int_0^1 \arctan^3(x)dx=\frac{3}{4}\big(M(\cbar1,1,\wc1)-M(\bar1,1,\wc1)\big),\\
&\int_0^1 \arctan^4(x)dx=\frac{3}{2}\big(M(\bar1,\wc1,1,\wc1)+M(\cbar1,\wc1,1,\wc1)\big),
\end{align*}
which are consistent with the evaluations in Example~\ref{eg:ArcTanPowerIntegral}.
\end{re}

Next, we derive the formulas for the moments of arc tangent powers over $[0,1]$.
For positive integers $m$ and $n$ such that $n\ge m$, we set
\begin{align*}
& T^{\bfgs_{2m-1}}_n:=T^{\bfgs_{2m-1}}_n(\{1\}_{2m-1}),
\quad T^{\bfgs_{2m}}_n:=T^{\bfgs_{2m}}_n(\{1\}_{2m}),\\
& S^{\bfgs_{2m-1}}_n:=S^{\bfgs_{2m-1}}_n(\{1\}_{2m-1}),
\quad S^{\bfgs_{2m}}_n:=S^{\bfgs_{2m}}_n(\{1\}_{2m}).
\end{align*}
For any $n\in \N$ we put $T_{n}^{\bfgs_0}=S_{n}^{\bfgs_0}:=1$.

\begin{thm}
For positive integers $n$ and $m$,
\begin{align}\label{even-even-xarctanx}
\int_0^1 x^{2n-2}\arctan^{2m}(x)dx %\nonumber\\
={}&\frac{1+(-1)^n}{2n-1}t^{2m}(\bar 1)+\frac{(-1)^{n+m}(2m)!}{2^{2m}(2n-1)}T_n^{\{1\}_{2m-1},-1}\nonumber\\
&-\frac{(-1)^{n}(2m)!}{2n-1}\sum_{u=0}^{m-1}\frac{(-1)^uT_n^{\{1\}_{2u}}}{(2m-2u)!2^{2u}}A(2m-2u)\nonumber\\
&+\frac{(-1)^{n}(2m)!}{2n-1}\sum_{v=1}^{m-1}\frac{(-1)^{v}t^{2m-2v}(\bar 1)}{2^{2v}(2m-2v)!}\left(T_n^{\{1\}_{2v}}+T_n^{\{1\}_{2v-1},-1}\right)\nonumber\\
&+\frac{(-1)^{n}(2m)!}{2n-1}\sum_{v=0}^{m-1}\frac{(-1)^{v}t^{2m-2v-1}(\bar 1)}{2^{2v+1}(2m-2v-1)!}\left(S_n^{\{1\}_{2v+1}}-S_n^{\{1\}_{2v},-1}\right),
\end{align}
\begin{align}\label{even-odd-xarctanx}
\int_0^1 x^{2n-2}\arctan^{2m-1}(x)dx% \nonumber\\
           ={}&-\frac{1+(-1)^n}{2n-1}t^{2m-1}(\bar 1)-\frac{(-1)^{n+m}(2m-1)!}{2^{2m-1}(2n-1)}S_n^{\{1\}_{2m-2},-1}\nonumber\\
            &-\frac{(-1)^{n}(2m-1)!}{2n-1}\sum_{u=0}^{m-1}\frac{(-1)^uT_n^{\{1\}_{2u}}}{(2m-2u-1)!2^{2u}}A(2m-2u-1)\nonumber\\
            &+\frac{(-1)^{n}(2m-1)!}{2n-1}\sum_{v=1}^{m-1}\frac{(-1)^{v+1}t^{2m-2v-1}(\bar 1)}{2^{2v}(2m-2v-1)!}\left(T_n^{\{1\}_{2v}}+T_n^{\{1\}_{2v-1},-1}\right)\nonumber\\
            &+\frac{(-1)^{n}(2m-1)!}{2n-1}\sum_{v=1}^{m-1}\frac{(-1)^{v}t^{2m-2v}(\bar 1)}{2^{2v-1}(2m-2v)!}\left(S_n^{\{1\}_{2v-1}}-S_n^{\{1\}_{2v-2},-1}\right),
\end{align}
\begin{align}\label{odd-even-xarctanx}
\int_0^1 x^{2n-1}\arctan^{2m}(x)dx %\nonumber\\
={}&\frac{1-(-1)^n}{2n}t^{2m}(\bar 1)+\frac{(-1)^{n+m}(2m)!}{2^{2m}(2n)}S_{n}^{\{1\}_{2m-1},-1}\nonumber\\
&+\frac{(-1)^{n}(2m)!}{2n}\sum_{u=0}^{m-1}\frac{(-1)^uT_{n}^{\{1\}_{2u+1}}}{(2m-2u-1)!2^{2u+1}}A(2m-2u-1)\nonumber\\
&+\frac{(-1)^{n}(2m)!}{2n}\sum_{v=0}^{m-1}\frac{(-1)^{v}t^{2m-2v-1}(\bar 1)}{2^{2v+1}(2m-2v-1)!}\left(T_{n}^{\{1\}_{2v+1}}+T_{n}^{\{1\}_{2v},-1}\right)\nonumber\\
&+\frac{(-1)^{n}(2m)!}{2n}\sum_{v=1}^{m-1}\frac{(-1)^{v+1}t^{2m-2v}(\bar 1)}{2^{2v}(2m-2v)!}\left(S_{n}^{\{1\}_{2v}}-S_{n}^{\{1\}_{2v-1},-1}\right),
\end{align}
\begin{align}\label{odd-odd-xarctanx}
\int_0^1 x^{2n-1}\arctan^{2m-1}(x)dx %\nonumber\\
={}&-\frac{1-(-1)^n}{2n}t^{2m-1}(\bar 1)+\frac{(-1)^{n+m}(2m-1)!}{2^{2m-1}(2n)}T_{n}^{\{1\}_{2m-2},-1}\nonumber\\
&-\frac{(-1)^{n}(2m-1)!}{2n}\sum_{u=1}^{m-1}\frac{(-1)^uT_{n}^{\{1\}_{2u-1}}}{(2m-2u)!2^{2u-1}}A(2m-2u)\nonumber\\
&+\frac{(-1)^{n}(2m-1)!}{2n}\sum_{v=1}^{m-1}\frac{(-1)^{v}t^{2m-2v}(\bar 1)}{2^{2v-1}(2m-2v)!}\left(T_{n}^{\{1\}_{2v-1}}+T_{n}^{\{1\}_{2v-2},-1}\right)\nonumber\\
&+\frac{(-1)^{n}(2m-1)!}{2n}\sum_{v=1}^{m-1}\frac{(-1)^{v}t^{2m-2v-1}(\bar 1)}{2^{2v}(2m-2v-1)!}\left(S_{n}^{\{1\}_{2v}}-S_{n}^{\{1\}_{2v-1},-1}\right),
\end{align}
where $A(p):=\int_0^1 \arctan^p(x)dx$, and it can be explicitly expressed by \eqref{arctan-cot} and Lemma \ref{lem-orr-2017}.
\end{thm}

\begin{proof}
Consider the integral
\begin{align*}
\int_0^1 x^{k-1}\arctan^p(x)dx=\frac1{k}\int_0^1 \arctan^p(x)dx^k %\nonumber \\
={}&\frac{(-1)^p t^p(\bar 1)}{k}-\frac{p}{k} \int_0^1\frac{x^k}{1+x^2}\arctan^{p-1}(x)dx,
\end{align*}
where we use the identity $\arctan(1)=-t(\bar 1)=\frac{\pi}{4}=-\sum_{n=1}^\infty \frac{(-1)^n}{2n-1}=\beta(1)$.
When $k=2n$ then
\begin{align}\label{k=2n}
\int_0^1 x^{2n-1}&\arctan^p(x)dx=\frac{(-1)^p t^p(\bar 1)}{2n}-\frac{p}{2n} \int_0^1\frac{x^{2n}}{1+x^2}\arctan^{p-1}(x)dx\nonumber\\
={}&\frac{(-1)^p t^p(\bar 1)}{2n}-\frac{p}{2n} \int_0^1 \left\{\frac{(-1)^n}{1+x^2}+(-1)^n\sum_{k=1}^n (-1)^k x^{2k-2}\right\}\arctan^{p-1}(x)dx\nonumber\\
={}&\frac{(-1)^p t^p(\bar 1)}{2n}-\frac{p}{2n} (-1)^n \int_0^1 \frac{\arctan^{p-1}(x)}{1+x^2}dx -\frac{p}{2n} (-1)^n \sum_{k=1}^n (-1)^k \int_0^1 x^{2k-2}\arctan^{p-1}(x)dx\nonumber\\
={}&\frac{1-(-1)^n}{2n} (-1)^p t^p(\bar 1)-p \frac{(-1)^n}{2n}\sum_{k=1}^n (-1)^k \int_0^1 x^{2k-2}\arctan^{p-1}(x)dx.
\end{align}
When $k=2n-1$ then
\begin{align}\label{k=2n-1}
&\int_0^1 x^{2n-2}\arctan^p(x)dx=\frac{(-1)^p t^p(\bar 1)}{2n-1}-\frac{p}{2n-1} \int_0^1\frac{x^{2n-1}}{1+x^2}\arctan^{p-1}(x)dx\nonumber\\
={}&\frac{(-1)^p t^p(\bar 1)}{2n-1}-\frac{p}{2n-1} \int_0^1 x\left\{\frac{(-1)^{n-1}}{1+x^2}+(-1)^{n-1}\sum_{k=1}^{n-1} (-1)^k x^{2k-2}\right\}\arctan^{p-1}(x)dx\nonumber\\
={}&\frac{(-1)^p t^p(\bar 1)}{2n-1}+p \frac{(-1)^n}{2n-1} \int_0^1 \frac{x\arctan^{p-1}(x)}{1+x^2}dx\ +p \frac{(-1)^n}{2n-1}\sum_{k=1}^{n-1} (-1)^k \int_0^1 x^{2k-1}\arctan^{p-1}(x)dx.
\end{align}
Integration by parts leads to
\begin{align*}
\int_0^1 \frac{x\arctan^{p-1}(x)}{1+x^2}dx={}&\int_0^1 x\arctan^{p-1}(x)d(\arctan(x))\\
={}&(-1)^p t^p(\bar 1)-\int_0^1 \arctan^{p}(x)dt-(p-1)\int_0^1 \frac{x\arctan^{p-1}(x)}{1+x^2}dx,
\end{align*}
and hence
\begin{align}\label{p-integral}
p\int_0^1 \frac{x\arctan^{p-1}(x)}{1+x^2}dx=(-1)^p t^p(\bar 1)-\int_0^1 \arctan^{p}(x)dx.
\end{align}
Plugging \eqref{p-integral} into \eqref{k=2n-1}, we obtain
\begin{align}\label{2,k=2n-1}
\int_0^1 x^{2n-2}\arctan^p(x)dx %\nonumber\\
={}&\frac{1+(-1)^n}{2n-1}(-1)^p t^p(\bar 1)-\frac{(-1)^n}{2n-1}\int_0^1 \arctan^{p}(x)dx\nonumber\\&\quad+p \frac{(-1)^n}{2n-1}\sum_{k=1}^{n-1} (-1)^k \int_0^1 x^{2k-1}\arctan^{p-1}(x)dx.
\end{align}
Using the two recurrence relations \eqref{k=2n} and \eqref{2,k=2n-1}, we get the desired evaluations immediately.
\end{proof}

\begin{thm}\label{thm-AMTV-AMSV-relation}
For any positive integer $m$,
\begin{align}
\label{arctan^2m/x}
\int_0^1 \frac{\arctan^{2m}(x)}{x}dx %\nonumber\\
={}&t^{2m-1}(\bar 1)(t(\bar{2})+t(2))+\frac{(-1)^{m}(2m-1)!}{2^{2m}}S(2,\{1\}_{2m-2},\bar{1})\nonumber\\
&\quad+(2m-1)!\sum_{u=0}^{m-1}\frac{(-1)^uA(2m-2u-1)}{(2m-2u-1)!2^{2u+1}}T(2,\{1\}_{2u})\nonumber\\
&\quad-(2m-1)!\sum_{v=1}^{m-1}\frac{(-1)^{v+1}t^{2m-2v-1}(\bar 1)}{2^{2v+1}(2m-2v-1)!}\left(T(2,\{1\}_{2v})+T(2,\{1\}_{2v-1},\bar{1})\right)\nonumber\\
&\quad-(2m-1)!\sum_{v=1}^{m-1}\frac{(-1)^{v}t^{2m-2v}(\bar 1)}{2^{2v}(2m-2v)!}\left(S(2,\{1\}_{2v-1})-S(2,\{1\}_{2v-2},\bar{1})\right)
,\\
\label{arctan^2m+1/x}
\int_0^1 \frac{\arctan^{2m+1}(x)}{x}dx % \nonumber\\
={}&-t^{2m}(\bar 1)(t(\bar{2})+t(2))-\frac{(-1)^{m}(2m)!}{2^{2m+1}}T(2,\{1\}_{2m-1},\bar{1})\nonumber\\
&\quad+(2m)!\sum_{u=0}^{m-1}\frac{(-1)^uA(2m-2u)}{(2m-2u)!2^{2u+1}}T(2,\{1\}_{2u})\nonumber\\
&\quad-(2m)!\sum_{v=1}^{m-1}\frac{(-1)^{v}t^{2m-2v}(\bar 1)}{2^{2v+1}(2m-2v)!}\left(T(2,\{1\}_{2v})+T(2,\{1\}_{2v-1},\bar{1})\right)\nonumber\\
&\quad-(2m)!\sum_{v=0}^{m-1}\frac{(-1)^{v}t^{2m-2v-1}(\bar 1)}{2^{2v+2}(2m-2v-1)!}\left(S(2,\{1\}_{2v+1})-S(2,\{1\}_{2v},\bar{1})\right),
\end{align}
where $t(\bar k)=-\beta(k)$ for $k\in \N$.
\end{thm}
\begin{proof}
Multiplying \eqref{even-even-xarctanx} and \eqref{even-odd-xarctanx} by $(-1)^{n-1}/(2n-1)$ and summing $n$ 
from 1 to infinity, we can deduce the desired results by 
\eqref{relations-odd-alternatingmultipleT-harmonicsums}--\eqref{relations-even-alternatingmultipleS-harmonicsums} 
and the definitions of AMSVs and AMTVs.
\end{proof}

\begin{exa}
Let $m=1,2$ in \eqref{arctan^2m/x} and \eqref{arctan^2m+1/x},
we can evaluate the following integrals precisely:
\begin{align*}
\int_0^1 \frac{\arctan^2(x)}{x}dx
={}&\tfrac{\pi}{4}G-\tfrac{\pi^2}{16}\log2-\tfrac{1}{4}S(2,\bar{1})
,\\
\int_0^1 \frac{\arctan^3(x)}{x}dx
={}&-\tfrac{\pi^2}{16}G+\tfrac{\pi^3}{32}\log2+\tfrac{1}{4}T(2,1,\bar{1})+\tfrac{\pi}{8}\big(S(2,1)-S(2,\bar{1})\big)
,\\
\int_0^1 \frac{\arctan^4(x)}{x}dx={}&\tfrac{63\pi^2}{512}\zeta(3)-\tfrac{5\pi^3}{64}G+\tfrac{3\pi^4}{256}\log2
+\tfrac{3}{8}S(2,1,1,\bar{1})+\tfrac{3\pi^2}{64}\big(S(2,1)-S(2,\bar{1})\big)\\
&+\tfrac{3\log 2}{8}T(2,1,1)+\tfrac{3\pi}{16}T(2,1,\bar{1})
,\\
\int_0^1 \frac{\arctan^5(x)}{x}dx={}&\tfrac{3\pi^2}{8}\beta(4)-\tfrac{9\pi^3}{512}\zeta(3)-\tfrac{11\pi^4}{256}G
+\tfrac{\pi^5}{256}\log2
-\tfrac{3\pi}{8}\big(S(2,1_3)-S(2,1,1,\bar{1})\big)\\
&
+\tfrac{\pi^3}{64}\big(S(2,1)-S(2,\bar{1})\big)-\tfrac{3}{4}T(2,1_3,\bar{1})
+\big(\tfrac{3}{2}G-\tfrac{3\pi}{8}\log2\big)T(2,1,1)+\tfrac{3\pi^2}{32}T(2,1,\bar{1}).
\end{align*}
\end{exa}

\begin{pro}\label{pro-int-arct-MTVs} For positive integer $r$,
\begin{align*}
\int_0^1 \frac{\arctan^r(x)}{x}dx=(-1)^{[(r+1)/2]}\frac{r!}{2^r} T({\bar 2},\{1\}_{r-1}).
\end{align*}
\end{pro}
\begin{proof}
According to the definition of $\arctan (x)$, we have
\begin{align*}
\arctan^r(x)=r! \int_0^x \left(\frac{dt}{1+t^2}\right)^r       %\nonumber\\
={}&r!\sum_{n_1>n_2>\cdots>n_r>0} \frac{(-1)^{n_1-r}x^{2n_1-r}}{(2n_1-r)(2n_2-r+1)\cdots (2n_r-1)}.
\end{align*}
Hence, multiplying it by $1/x$ and integrating over $(0,1)$ yields
\begin{align*}
\int_0^1 \frac{\arctan^r(x)}{x}dx=r!\sum_{n_1>n_2>\cdots>n_r>0} \frac{(-1)^{n_1-r}}{(2n_1-r)^2(2n_2-r+1)\cdots (2n_r-1)}.
\end{align*}
Further, from the definition of AMTVs \eqref{defn-AMTVs} give
\begin{align*}
T({\bar 2},\{1\}_{2m-2}) %\nonumber\\
={}&(-1)^{m+1} 2^{2m-1} \!\!\! \sum_{n_1>n_2>\cdots>n_{2m-1}>0} \frac{(-1)^{n_1}}{(2n_1-2m+1)^2(2n_2-2m+2)\cdots (2n_{2m-1}-1)},\\
T({\bar 2},\{1\}_{2m-1})    %\nonumber\\
={}&(-1)^{m} 2^{2m} \!\!\! \sum_{n_1>n_2>\cdots>n_{2m-1}>0} \frac{(-1)^{n_1}}{(2n_1-2m)^2(2n_2-2m+1)\cdots (2n_{2m}-1)}.
\end{align*}
Thus, we obtain the desired result with an elementary calculation.
\end{proof}

From \cite[Props. 3.21 and 3.22]{XuZhao2020b}, for $p\in \N$ we have
\begin{align}\label{equ:T2bar1p}
T(\bar 1,\{1\}_{p-1},\bar 1)=(-1)^pT(\ol{2},\{1\}_{p-1})\in \mathbb{Q}[\beta(1),\zeta(2),\beta(2),\zeta(3),\beta(3),\zeta(4),\ldots].
\end{align}
Hence, we deduce the following evaluations ($\beta(2)=G$):
\begin{align*}
&\int_0^1 \frac{\arctan(x)}{x}dx=G,\\
&\int_0^1 \frac{\arctan^2(x)}{x}dx=\frac{1}{2}\pi G-\frac7{8}\zeta(3),\\
&\int_0^1 \frac{\arctan^3(x)}{x}dx=\frac9{8}\zeta(2)G-\frac{3}{2}\beta(4),\\
&\int_0^1 \frac{\arctan^4(x)}{x}dx=\frac{93}{32}\zeta(5)-\frac3{2}\pi \beta(4)+\frac1{16}\pi^3G.
\end{align*}

\begin{cor} \label{cor:ST21s}
For positive integer $m$,
\begin{align*}
S(2,\{1\}_{2m-2},{\bar 1}), T(2,\{1\}_{2m-1},{\bar 1})\in \mathbb{Q}[\log2,\pi,\zeta(2),\beta(2),\zeta(3),\beta(3),\zeta(4),\ldots].
\end{align*}
\end{cor}
\begin{proof}
The corollary follows immediately from Theorem \ref{thm-AMTV-AMSV-relation} and Proposition \ref{pro-int-arct-MTVs} and
notating the fact that $S(2,\{1\}_{2m-1})$ and $T(2,\{1\}_{m-1})=T(m+1)$ can be expressed in terms of a linear combination
of products of Riemann zeta values (see \cite[Eq. (3.17)]{XuZhao2020a}).
\end{proof}

The following four examples provide precise evaluations for those appearing in Corollary \ref{cor:ST21s}:
\begin{align*}
&S(2,{\bar 1})=\frac{7}{2}\zeta(3)-\pi G-\frac{\pi^2}{4}\log2,\\
&T(2,1,{\bar 1})=-6\beta(4)+3\zeta(2)G,\\
&S(2,1,1,{\bar 1})=\frac{31}{4}\zeta(5)-\frac{15}{8}\zeta(4)\log2-\frac{63}{32}\zeta(2)\zeta(3)-\pi \beta(4),\\
&T(2,1,1,1,{\bar 1})=\frac{15}{4}\zeta(4)G+3\zeta(2)\beta(4)-10\beta(6).
\end{align*}

\medskip\noindent
\textbf{Question 1. }  From the second and fourth examples, we can find that $\log2$ does not appear. Is this always the case, i.e,
\begin{align*}
T(2,\{1\}_{2m-1},{\bar 1})\in \mathbb{Q}[\pi,\zeta(2),\beta(2),\zeta(3),\beta(3),\zeta(4),\ldots] ?
\end{align*}

\begin{re}
In \cite{Cha2024}, S. Charlton provides an affirmative answer to Question 1 by giving an explicit generating series 
identity to evaluate these special type alternating multiple $T$-values. These values also appear
predominantly in \cite{KanekoTs2022}.
\end{re}

\begin{re} (i) Recall that the convoluted $T$-values $T(\bfk \circledast \bfl)$
(see \cite[Defn. 1.2]{XuZhao2020a}) can be regarded as a $T$-variant of Kaneko--Yamamoto MZVs $\zeta({\bfk}{\circledast} {\bfl}^\star)$ (see \cite{KY2018}). Similarly, by using the alternating multiple $T$-harmonic sums and alternating multiple $S$-harmonic sums, it is possible to define the alternating convoluted $T$-values so that the first factor in the sum is alternating multiple $T$-harmonic sums and the second is either alternating multiple $S$-harmonic sums or alternating multiple $T$-harmonic sums. (ii) For positive integers $k$ and $m$, it is possible to establish some explicit relations between $S(k+1,\{1\}_{2m-2},{\bar 1})$ and $T(k+1,\{1\}_{2m-1},{\bar 1})$ (even more general alternating convoluted $T$-values) by considering the following iterated integral
\begin{align*}
\int_0^1 \frac{\arctan^p(t)dt}{t}\left(\frac{dt}{t}\right)^m\frac{\arctan^q(t)dt}{t}\quad (p,m\in\N_0,q\in\N).
\end{align*}
Clearly, the above iterated integral can be expressed in terms of AMMVs.
\end{re}

Applying Au's Mathematica package \cite{Au2020,Au2022}, we get many explicit evaluations of AMSVs and AMTVs with arguments that are more than 1 and 2. Unfortunately, we cannot find any patterns. Some cases are listed below.

\begin{exa} \label{Exa-2} We have
\begin{align*}
&S(\bar3,1,\bar1)=2 \Li_5(\tfrac{1}{2})-\tfrac{589}{256}\zeta(5)-\tfrac{7}{8} \zeta (3) \log^22-\tfrac{1}{60}
    \log^52+\tfrac{\pi^2}{36}  \log^32 +\tfrac{151\pi^4}{5760} \log2,\\
&S(4,1,\bar1)=\zeta(\bar5,1)+\tfrac{\pi^2}{12}\Li_4(\tfrac{1}{2})-\tfrac{83}{128} \zeta^2(3)+\tfrac{7\pi^2}{32} \zeta(3) \log2-\tfrac{31}{16}\zeta (5)\log2+\tfrac{227 \pi ^6}{967680}+\tfrac{\pi^2}{288}\log^42-\tfrac{\pi^4}{288}\log^2 2,\\
&T(\bar3,1,\bar1)=-\tfrac{1}{8}\pi ^3 G-\tfrac{7}{32}  \pi ^2 \zeta (3)+\tfrac{93}{16} \zeta (5),\\
&T(\bar3,2,\bar1)=12 G\beta(4)-\tfrac{1}{2}\pi^2 G^2-\tfrac{\pi^3}{8} G \log2-\tfrac{\pi^3 }{4} \Im\Li_3(\tfrac{1+i}{2})+\pi^2 \Li_4(\tfrac{1}{2}) \log2\\
&\quad\quad\quad\quad\quad+\tfrac{7\pi^2}{8}  \zeta(3) -\tfrac{257\pi}{23040} ^6+\tfrac{\pi^2}{24}  \log ^4(2)-\tfrac{13\pi^4}{384} \log^2 2.
\end{align*}
\end{exa}
\begin{re} 
For explicit evaluations of more general (alternating) triple $t$- and $T$-values,  see \cite{XuWang2020}.
\end{re}

The following three theorems provide the evaluations of AMSVs and AMTVs of special forms.

\begin{thm} For $m\in\N_0$, we have
\begin{align*}
{}& S(\bar1,\{1\}_m)=(-1)^{[(m+1)/2]}\frac{\pi}{2}\Big)^{m} \sum_{\ell=0}^m \frac{(-2)^{\ell+1}}{\ell!(m-\ell)!}  \\
&\times \left(\left(\frac1{2^{\ell+1}}-\delta_{0,\ell}\right) \log2
 -\sum_{k=1}^{\ell}(-1)^k\binom{\ell}{k}\left\{\left(1-\frac1{2^{k+1}}\right)\sum_{j=1}^{[k/2]} \frac{k!(-1)^j(4^j-1)}{(k-2j)!(2\pi)^{2j}}\zeta(2j+1)\atop +\frac1{2^{k+1}}\sum_{j=1}^{[(k+1)/2]} \frac{k!(-4)^j\beta(2j)}{(k+1-2j)!\pi^{2j-1}}\right\} \right)  \\
\in {}& \mathbb{Q}[\log2,\pi,\zeta(2),\beta(2),\zeta(3),\beta(3),\zeta(4),\ldots].
\end{align*}
\end{thm}
\begin{proof}
According to the iterated integral expressions of AMSVs, one obtains
\begin{align*}
S(\bar1,\{1\}_m)={}&(-1)^{[m/2]}\int_0^1 \underbrace{w_{-1}^{-1}\cdots w_{-1}^{-1}}_{m}w_{-1}^{+1}\\
={}&(-1)^{[(m+1)/2]}2^{m+1}\int_0^1 \left(\frac{dt}{1+t^2}\right)^{m}\frac{tdt}{1+t^2}\\
={}&\frac{(-1)^{[(m+1)/2]}2^{m+1}}{m!}\int_0^1 \frac{\Big(\frac{\pi}{4}-\arctan t\Big)^m tdt}{1+t^2}\\
={}&\frac{(-1)^{[(m+1)/2]}2^{m+1}}{m!}\sum_{\ell=0}^m \binom{m}{\ell} \Big(\frac{\pi}{4}\Big)^{m-\ell}(-1)^\ell
    \int_0^1 \frac{(\arctan t)^\ell t}{1+t^2}dt\\
={}&(-1)^{[(m+1)/2]}2^{m+1} \sum_{\ell=0}^m \frac{\Big(\frac{\pi}{4}\Big)^{m-\ell}(-1)^\ell}{(\ell+1)!(m-\ell)!}
    \left\{\Big(\frac{\pi}{4}\Big)^{\ell+1}-\int_0^1 (\arctan t)^{\ell+1}dt \right\},
\end{align*}
where $[x]$ denotes the greatest integer less than or equal to $x$.
Proposition \ref{pro-int-arc} now yields the desired result immediately.
\end{proof}

\begin{thm} \label{thm:T111s}
For $a,b\in\N_0$, we have
\begin{align*}
&T(\bar1,\{1\}_a,2,\{1\}_b)\in \mathbb{Q}[\pi,\zeta(2),\beta(2),\zeta(3),\beta(3),\zeta(4),\ldots].
\end{align*}
\end{thm}
\begin{proof}
Applying the iterated integral expression of AMTVs gives
\begin{align*}
T(\overline{k_1},k_2,\ldots,k_r)=(-1)^{[r/2]}\int_0^1 w_0^{k_1-1}w_{-1}^{-1}w_0^{k_2-1}w_{-1}^{-1}\ldots w_0^{k_r-1}w_{-1}^{-1}.
\end{align*}
Hence, by an elementary calculation yields
\begin{align*}
&T(\bar1,\{1\}_a,2,\{1\}_b)=(-1)^{[(a+b+2)/2]}\int_0^1 \underbrace{w_{-1}^{-1}\cdots w_{-1}^{-1}}_{a+1}w_0 \underbrace{w_{-1}^{-1}\cdots w_{-1}^{-1}}_{b+1}\\
={}&2^{a+b+2}(-1)^{[(3a+3b)/2]+1}\int_0^1 \left(\frac{dt}{1+t^2}\right)^{a+1}\frac{dt}{t} \left(\frac{dt}{1+t^2}\right)^{b+1}\\
={}&\frac{2^{a+b+2}(-1)^{[(3a+3b)/2]+1}}{(a+1)!(b+1)!}\int_0^1 \frac{\Big(\frac{\pi}{4}-\arctan t \Big)^{a+1}(\arctan t)^{b+1}}{t}dt\\
={}&\frac{2^{a+b+2}(-1)^{[(3a+3b)/2]+1}}{(a+1)!(b+1)!}\sum_{k=0}^{a+1}\binom{a+1}{k}\Big(\frac{\pi}{4}\Big)^{a+1-k}(-1)^k\int_0^1 \frac{(\arctan t)^{k+b+1}}{t}dt.
\end{align*}
Thus, using Proposition \ref{pro-int-arct-MTVs} and \eqref{equ:T2bar1p} we see that
\begin{align*}
\int_0^1 \frac{(\arctan t)^{m}}{t}dt\in T(\ol{2},\{1\}_{m-1})\Q \subset \mathbb{Q}[\beta(1),\zeta(2),\beta(2),\zeta(3),\beta(3),\zeta(4),\ldots]\quad (m\in\N).
\end{align*}
This yields the desired conclusion at once.
\end{proof}

\begin{thm} \label{thm:S111s}
For any $m\in\N_0$, we have
\begin{align*}
S(\bar1,\{1\}_m,\bar1)\in \mathbb{Q}[\log2,\pi,\zeta(2),\beta(2),\zeta(3),\beta(3),\zeta(4),\ldots].
\end{align*}
\end{thm}
\begin{proof} Similar to the iterated integral of $S(\bar1,\{1\}_m)$, by definition of AMSVs, we obtain the iterated integral of $S(\bar1,\{1\}_m,\bar1)$:
\begin{align*}
S(\bar1,\{1\}_m,\bar1)=\frac{(-1)^{[m/2]+1}2^{m+2}}{(m+1)!} \int_0^1 \frac{\Big(\frac{\pi}{4}-\arctan t\Big)^{m+1}t}{1-t^2}dt \quad (m\in\N_0).
\end{align*}
Integration by parts yields
\begin{align*}
S(\bar1,\{1\}_m,\bar1)=\frac{(-1)^{[m/2]}2^{m+1}}{m!} \int_0^1 \frac{\Big(\frac{\pi}{4}-\arctan t\Big)^{m}\log(1-t^2)}{1+t^2}dt.
\end{align*}
Setting $t=\tan x$ gives
\begin{align*}
S(\bar1,\{1\}_m,\bar1)={}&\frac{(-1)^{[m/2]}2^{m+1}}{m!} \int_0^{\pi/4} \Big(\frac{\pi}{4}-x\Big)^m\log\left(\frac{\cos(2x)}{\cos^2(x)}\right)dt\\
={}&\frac{(-1)^{[m/2]}2^{m+1}}{m!} \int_0^{\pi/4} \Big(\frac{\pi}{4}-x\Big)^m\log(\cos(2x))dt\\
&\quad-\frac{(-1)^{[m/2]}2^{m+2}}{m!} \int_0^{\pi/4} \Big(\frac{\pi}{4}-x\Big)^m\log(\cos(x))dt.
\end{align*}
Letting $x=\frac{\pi}{4}-u$ in the first integral and $x=\frac{\pi}{2}-u$ in the second integral, one obtains
\begin{align*}
S(\bar1,\{1\}_m,\bar1)={}&\frac{(-1)^{[m/2]}2^{m+1}}{m!} \left\{\int_0^{\pi/4} u^m\log(\sin 2u)du-2\int_{\pi/4}^{\pi/2} \Big(u-\frac{\pi}{4}\Big)^m\log(\sin u) du\right\}\\
={}&\frac{(-1)^{[m/2]}}{m!} \int_0^{\pi/2} y^m\log(\sin y)dy-\frac{(-1)^{[m/2]}2^{m+2}}{m!}\sum_{k=0}^m \binom{m}{k}\left(\frac{\pi}{4}\right)^{m-k}\int_{\pi/4}^{\pi/2}u^k\log(\sin u) du,
\end{align*}
where we replaced the $2u$ by $y$ in the last step.

On the other hand, applying Lemma \ref{lem-orr-2017} and using the fact that for $p\in \N$,
\begin{align*}
\int_0^{\pi/2} x^p\cot(x)dx=-p \int_0^{\pi/2} x^{p-1}\log(\sin x)dx\in \mathbb{Q}[\log2,\pi,\zeta(2),\zeta(3),\zeta(4),\zeta(5),\ldots]
\end{align*}
and
\begin{align*}
\int_0^{\pi/4} x^p\cot(x)dx=-p \int_0^{\pi/4} x^{p-1}\log(\sin x)dx\in \mathbb{Q}[\log2,\pi,\zeta(2),\beta(2),\zeta(3),\beta(3),\zeta(4),\ldots],
\end{align*}
we deduce the desired result immediately.
\end{proof}

\begin{exa} \label{Exa-3} The following identities provide precise evaluations of a few
AMMVs that have appeared in Theorem~\ref{thm:T111s} and Theorem~\ref{thm:S111s}:
	\begin{alignat*}{4}
%		&S(\bar1,1)=-2G+\pi\log2,   \\
%		&S(\bar1,\bar1)=-2G+\tfrac1{2}\pi\log2,   \\
		&S(\bar1,1,1)=\tfrac1{4}\pi^2\log2-\tfrac{21}{16}\zeta(3),   &&S(\bar1,1,1,1)=-2\beta(4)-\tfrac1{24}\pi^3\log2+\tfrac{3}{4}\pi\zeta(3),\\
		&S(\bar1,1,\bar1)=\tfrac1{8}\pi^2\log2-\tfrac7{8}\zeta(3),   &&S(\bar1,1,1,\bar1)=-2\beta(4)-\tfrac1{48}\pi^3\log2+\tfrac{21}{32}\pi\zeta(3),\\
		&T(\bar1,2)=G\pi-\tfrac7{2}\zeta(3),    &&T(\bar1,2,1,1)=-\tfrac{1}{24}G\pi^3+3\pi\beta(4)-\tfrac{31}{4}\pi\zeta(5), \\
		&T(\bar1,2,1)=-\tfrac{1}{4}G\pi^2+6\beta(4)-\tfrac7{8}\pi\zeta(3), \quad \    &&T(\bar1,1,1,2)=\pi\beta(4)+\tfrac7{16}\pi^2\zeta(3)-\tfrac{31}{4}\zeta(5)\\
		&T(\bar1,1,2)=-6\beta(4)+\tfrac7{4}\pi\zeta(3), &&T(\bar1,1,2,1)=-3\pi\beta(4)-\tfrac7{32}\pi^2\zeta(3)+\tfrac{93}{8}\zeta(5). 	
%		 &T(\bar1,2,1,1,1)=\tfrac{1}{192}G\pi^4-\tfrac3{4}\pi^2\beta(4)+10\beta(6)-\tfrac{31}{32}\pi\zeta(5),\\
%		&T(\bar1,1,1,1,2)=-10\beta(6)-\tfrac7{96}\pi^3\zeta(3)+\tfrac{31}{8}\zeta(5),\\
%		&T(\bar1,1,2,1,1)=\tfrac3{4}\pi^2\beta(4)-20\beta(6)+\tfrac{31}{8}\pi\zeta(5),\\
%		 &T(\bar1,1,1,2,1)=-\tfrac1{4}\pi^2\beta(4)+20\beta(6)+\tfrac7{192}\pi^3\zeta(3)-\tfrac{93}{16}\pi\zeta(5),\\
%		 &S(\bar1,1,1,1,1)=-\tfrac1{192}\pi^4\log2+\tfrac{3}{16}\pi^2\zeta(3)-\tfrac{465}{256}\zeta(5),\\
%		 &S(\bar1,1,1,1,\bar1)=-\tfrac1{384}\pi^4\log2+\tfrac{21}{128}\pi^2\zeta(3)-\tfrac{217}{128}\zeta(5),\\
%		 &S(\bar1,1_5)=-2\beta(6)+\tfrac1{1920}\pi^5\log2-\tfrac1{32}\pi^3\zeta(3)+\tfrac{15}{16}\pi\zeta(5),\\
%		 &S(\bar1,1_4,\bar1)=-2\beta(6)+\tfrac1{3840}\pi^5\log2-\tfrac7{256}\pi^3\zeta(3)+\tfrac{465}{512}\pi\zeta(5).
	\end{alignat*}
\end{exa}

Moreover, by using numerical and symbolic computation with Mathematica we can obtain many explicit evaluations of AMTVs
with arguments of the form $(\bar2,1,\ldots,1)$ and AMSVs with arguments of the form
$(\bar2,1,\ldots,1)$ and $(\bar2,1,\ldots,1,\bar1)$.

\begin{exa}\label{Exa-1}
By applying Au's Mathematica package \cite{Au2020,Au2022} and noting that
\begin{align*}
	\Li_4\Big(\frac{1}{2}\Big)=\frac12\zeta(\bar3,1)+\frac1{96}\pi^4+\frac1{24}\pi^2\log^2 2-\frac1{24}\log^4 2-\frac7{8}\zeta(3)\log2,
\end{align*}
we can obtain the following evaluations:
	\begin{align*}
		&T(\bar 2,1,\bar 1)=-4G^2+\tfrac1{32}\pi^4,\\
		&T(\bar 2,1,1,1,\bar 1)=\tfrac{\pi^2}{2}G^2+\tfrac{5\pi^6}{1536}-8G\beta(4),\\
		 &S(\bar2,1,1)=-2G^2+\tfrac{7\pi^4}{1440}+2G\pi\log2+2\zeta(\bar3,1)-\tfrac{7}{2}\zeta(3)\log2,\\
		 &S(\bar2,1,\bar1)=-2G^2+\tfrac{59}{5760}\pi^4+G\pi\log2+\zeta(\bar3,1)-\tfrac{7}{4}\zeta(3)\log2,\\
		 &S(\bar2,1_4)=\tfrac{31\pi^6}{20160}-4G\beta(4)-\tfrac{\pi^3}{12}G\log2+2\pi\beta(4)\log2+2\zeta(\bar5,1)+\tfrac{3\pi}{2}G\zeta(3)
-\tfrac{33}{16}\ze^2(3)-\tfrac{31}{8}\zeta(5)\log2,\\
 &S(\bar2,1_3,\bar1)=\tfrac{443\pi^6}{322560}-4G\beta(4)-\tfrac{\pi^3}{24}G\log2+\pi\beta(4)\log2+\zeta(\bar5,1)+\tfrac{21\pi}{16}G\zeta(3)
-\tfrac{195}{128}\ze^2(3)-\tfrac{31}{16}\zeta(5)\log2.
	\end{align*}
\end{exa}

We would like to conclude this section with the following question.

\medskip\noindent
\textbf{Question 2.}  Is it true that for $m\in \N$,
\begin{align*}
	&T(\bar 2,\{1\}_{2m-1},\bar 1)\in \mathbb{Q}[\pi,\zeta(2),\beta(2),\zeta(3),\beta(3),\zeta(4),\ldots]
\end{align*}
and $S(\bar2,\{1\}_{2m})$ and $S(\bar2,\{1\}_{2m-1},\bar1)$ can be expressed in terms of a linear combinations of products of $\log2,\pi,\zeta(2l+1),\beta(2p)$ and $\zeta(\overline{2k+1},1)$, where $l,p,k\in \N$?

\begin{re}
S.\ Charlton \cite{Cha2024} has given affirmative answers of the two claims about $T(\bar 2,\{1\}_{2m-1},\bar 1)$ 
and $S(\bar2,\{1\}_{2m})$, and showed that $S(\bar2,\{1\}_{2m-1},\bar1)$ can be expressed in terms of a linear combinations 
of products of $\log2,\pi,\zeta(2l+1),\beta(2p)$ and $\zeta(\overline{2k_1+1},2k_2+1)$, where $l,p\in \N,k_1,k_2\in \N_0$.
\end{re}

\section{Dimension computation of AMMVs}\label{DCAMMVs}
Recall that $\AMMV_w$ is the $\Q$-vector space generated by all the AMMVs of weight $w$ and similarly for all of its subspaces. 
We now consider two more subspaces of $\AMMV$. Let $\AMMVe_w$  (resp.\ $\AMMVo_w$) be the $\Q$-span of all
AMMVs with the smallest summation index being even (resp.\ odd).
Set $\dim_\Q V_0=1$ for all the subspaces $V$ of $\AMMV$ including $\AMMV$ itself. Clearly, AMtVs, AMSVs and AMTVs are special cases of AMMVs, and the AMMVs can be written as $\Q[i]$-linear combinations of the CMZVs of level four. In \cite{Au2022}, using proven relations, Au reduced all CMZVs of level four, weight $\leq 6$ to expected number of generators (as per the standard motivic CMZV conjectures).

Applying Au's Mathematica package \cite{Au2020,Au2022}, we have found the data in Table \ref{Table:dimAMMV}.
\begin{table}[!h]
{
\caption{Conjectural dimensions of various subspaces of $\AMMV$.}
\begin{center}\label{Table:dimAMMV}
\begin{tabular}{  |c|c|c|c|c|c|c|c|} \hline
       $w$         & 0 &  1  &  2  &  3  &  4  &  5  &  6 \\ \hline
$\dim_\Q \AMZV_w$  & 1 &  1  &  2  &  3  & 5  &  8  &  13 \\ \hline
$\dim_\Q \AMtV_w$  & 1 &  1  &  3  &  6  & 12  &  24  &  48 \\ \hline
$\dim_\Q \AMSV_w$  & 1 &  1  &  3  &  6  &  12  &  22  &  42  \\ \hline
$\dim_\Q \AMTV_w$  & 1 &  1  &  2  &  4  & 7  &  13  &  24  \\ \hline
$\dim_\Q \AMMV_w$  & 1 &  2  &  4  &  8  &  16  &   32 &  64 \\ \hline
$\dim_\Q \AMMVe_w$  & 1 &  2  &  4  &  8  &  16  &   32 &  64 \\ \hline
$\dim_\Q \AMMVo_w$  & 1 &  1  &  3  &  7  &  15  &   31 &  63 \\ \hline
$\dim_\Q \CMZV^4_w$ & 1 & 2  &  4  &  8  &   16 &  32  &  64 \\ \hline
\end{tabular}
\end{center}
}
\end{table}

\begin{re} (a). The above conjectural dimensions provide genuine upper bounds because all the relations used to obtain
them in \cite{Au2022} have been proved previously.
(b). In the database associated with reference \cite{BlumleinBrVe2010}, Bl\"umlein et al. constructed explicit conjectural basis for the $\Q$-vector space generated by alternating multiple zeta values up to weight 12.
\end{re}

We have also computed the upper bound of the dimensions numerically with GP-PARI by using the rational linear relation
detection program LLL. We then obtained the following conjecture extending the above table:
\begin{table}[!h]
{
\caption{Conjectural dimensions of various subspaces of $\AMMV$.}
\begin{center}\label{Table:dimAMMV2}
\begin{tabular}{  |c|c|c|c|c|c|c|c|c|c|} \hline
       $w$         & 0 &  1  &  2  &  3  & 4  &  5  &  6  & 7  & 8\\ \hline
$\dim_\Q \AMZV_w$  & 1 &  1  &  2  &  3  & 5  &  8  &  13 & 21 & 34\\ \hline
$\dim_\Q \AMtV_w$  & 1 &  1  &  3  &  6  & 12 &  24 &  48 & 96 & 192 \\ \hline
$\dim_\Q \AMSV_w$  & 1 &  1  &  3  &  6  & 12 &  22 &  42 & 80 & 156 \\ \hline
$\dim_\Q \AMTV_w$  & 1 &  1  &  2  &  4  & 7  &  13 &  24 & 44 & 81 \\ \hline
\end{tabular}
\end{center}
}
\end{table}

For each weight $w$, we let $\bfMB_w$ be a conjectural basis of $\Q$-vector space generated by all AMMVs of weight $w$. Using Au's Mathematica package \cite{Au2020,Au2022}, we obtain
\begin{align*}
&\bfMB_1=\{\pi,\log2\},\\
&\bfMB_2=\{G,\pi^2,\pi\log2,\log^2 2\},\\
&\bfMB_3=\left\{\pi^3,\pi^2\log2,\pi\log^2 2,\log^3 2,\pi G,G\log2,\zeta(3),\Im\Li_3\Big(\frac{1+i}{2}\Big)\right\},\\
&\bfMB_4=\left\{\pi^4,\pi^3\log2,\pi^2\log^2 2,\pi\log^3 2,\log^4 2,\pi^2G,\pi G\log2,G\log^2 2,G^2,\pi\zeta(3),\atop
\zeta(3)\log2,\pi\Im\Li_3\big(\tfrac{1+i}{2}\big),\log2\Im\Li_3\big(\tfrac{1+i}{2}\big),
\beta(4),\Li_4\Big(\frac{1}{2}\Big),\Im\Li_4\big(\tfrac{1+i}{2}\big)
\right\}.
\end{align*}
From numerical computations we now propose a possible basis for AMMVs.
\begin{con}\label{conj:AMMVbasis}
For any $w\in\N$ the following set is a basis of $\AMMV_w$:
\begin{align*}
\bfMB_w=\Big\{M_\bfgs^\bfeps(\{1\}_w): \bfgs=(-1,\{1\}_{w-1}), \bfeps\in\{\pm 1\}^w \Big\}.
\end{align*}
\end{con}

This was in fact motivated by the following conjecture by the third named author concerning the
structure of CMZVs of level 3 and 4. Recall that the \emph{unit} CMZVs of level $N$ have the form
\begin{equation*}
\Li_{\{1\}^r}(\mu_1,\dots,\mu_r):=\sum_{n_1>\cdots>n_r>0}
\frac{\mu_1^{n_1}\dots \mu_r^{n_r}}{n_1 \dots n_r },
\end{equation*}
by \eqref{equ:defnMPL}, where $\mu_j^N=1$ for all $j=1,\dots,r$. 

\begin{con}\label{conj:CMVZ4} \emph{(\cite[Conjecture 5.2]{Zhao2010b})}
Let $\CMZV_N^w$ be the $\Q$-span of all CMZVs of level $N$ and weight $w$.
Then for $N=3$ and $N=4$, the following set of unit CMZVs is a basis of $\CMZV_N^w$:
\begin{align*}
\Big\{\Li_{\{1\}^w}(\mu_1,\dots,\mu_w):   \mu_j\in\{\exp(2\pi i/N),\exp(4\pi i/N)\} \ \forall 1\le j\le w  \Big\}.
\end{align*}
\end{con}

\begin{re}
Assuming Grothendieck's period conjecture, Deligne showed in \cite{Deligne2010} that $\dim_\Q \CMZV_N^w=2^w$
for $N=3$ and 4, which is still true if we replace $\Q$ by $\Q[i]$.
Moreover, J.\ Li showed in \cite[Theorem~1.2]{LiJiangtao2024}
that unit CMZVs of level $N$ and weight $w$ form a generating set of $\CMZV_N^w$ for $N=3,4$.
Therefore, by the proof of Theorem~\ref{thm:AMMV=CMZV4}, the space $\AMMV_w$ is generated by
\textbf{unit} AMMVs of weight $w$ (namely, those AMMVs of the form $M_\bfgs^\bfeps(\{1\}^w)$).
\end{re}

We now turn to special subspaces of $\AMMV$. For convenience, let
\begin{align*}
\tilde{t}(k_1,\ldots,k_r;\sigma_1,\ldots,\sigma_r):=M_{\sigma_1,\ldots,\sigma_r}^{\{-1\}_r}(k_1,\ldots,k_r)=2^r t(k_1,\ldots,k_r;\sigma_1,\ldots,\sigma_r),
\end{align*}
where $(k_1,\sigma_1)\neq (1,1)$ and the AMtVs are denoted by
\begin{align*}
t(k_1,\ldots,k_r;\sigma_1,\ldots,\sigma_r):=\sum_{n_1>\ldots>n_r>0} \frac{\sigma_1^{n_1}\cdots\sigma_r^{n_r}}{(2n_1-1)^{k_1}\cdots(2n_r-1)^{k_r}}.
\end{align*}
In particular, $t(k_1,\ldots,k_r)=t(k_1,\ldots,k_r;\{1\}_r)$. Similar to AMZVs, we may compactly indicate the presence of 
an alternating sign by placing a bar over the corresponding exponent $k_j$ when $\sigma_j=-1$. For example,
\begin{equation*}
\tilde{t}(\bar 2,3,\bar 1,4)=\tilde{t}(2,3,1,4;-1,1,-1,1).
\end{equation*}

Using Au's Mathematica package, we have found some conjectural basis $\bftB_w$ (resp.\ $\bfTB_w$, resp.\ $\bfSB_w$)
of the $\Q$-span by all AMtVs (resp.\ AMTVs, resp.\ AMSVs) of weight $w$ with $1\leq w\leq 6$:
\begin{align*}
&\textbf{tB}_1=\left\{t(\bar1)\right\},\\
&\textbf{tB}_2=\left\{t(2),t(\bar2),t(\bar1,1)\right\},\\
&\textbf{tB}_3=\left\{t(3),t(\bar3),t(2,1),t(2,\bar1),t(\bar1,\bar2),t(\bar1,1,1)\right\},\\
&\textbf{tB}_4=\left\{\begin{array}{l}
t(4),t(\bar4),t(3,1),t(3,\bar1),t(\bar3,1),t(2,\bar2),t(\bar2,\bar2),
t(2,1,1),t(2,1,\bar1),t(\bar1,2,\bar1),t(\bar1,1,\bar2),t(\bar1,1_3)\end{array}\right\},\\
&\textbf{tB}_5=\left\{\begin{array}{l}      t(5),t(\bar5),t(4,1),t(4,\bar1),t(\bar4,1),t(\bar4,\bar1),t(\bar1,\bar4),t(3,\bar2),t(\bar3,\bar2),
\\
t(\bar3,1_2),t(3,1_2),t(3,1,\bar1),t(3,\bar1,1),t(2_2,\bar1),t(\bar2,2,\bar1),t(2,1,\bar2),t(2,\bar1,\bar2),t(\bar1,\bar2_2),\\
t(2,1_3),t(2,1,1,\bar1),t(\bar1,2,1,\bar1),t(\bar1,1,2,\bar1),t(\bar1,1,1,\bar2),t(\bar1,1_4)\end{array}\right\},\\
&\textbf{tB}_6=\left\{\begin{array}{l} t(6),t(\bar6),t(5,1),t(5,\bar1),t(\bar5,1),t(\bar5,\bar1),t(\bar1,5),t(4,\bar2),t(\bar4,\bar2),t(2,\bar4),t(\bar2,\bar4),t(4,1_2),t(4,1,\bar1),\\ t(4,\bar1,1),t(4,\bar1_2),t(\bar4,1_2),t(\bar4,1,\bar1),t(\bar1,4,\bar1),t(\bar1,\bar4,1),t(\bar1,1,\bar4),t(3,\bar1,\bar2),t(\bar3,1,\bar2),t(3,1,\bar2),\\
t(\bar2_3),t(2,\bar2_2),t(\bar3,2,\bar1),t(3,2,\bar1),t(3,1,\bar1,1),t(3,1_2,\bar1),t(3,\bar1,1_2),t(3,1_3),t(\bar3,1_3),\\
t(2_2,1,\bar1),t(\bar2,2,1,\bar1),t(2,1,2,\bar1),t(\{2,\bar1\}_2),t(2,1,1,\bar2),t(2,1,\bar1,\bar2),t(\bar1,\bar2,2,\bar1),
    t(\bar1,2,\bar1,\bar2),\\
t(\bar1,1,\bar2_2),t(2,1_4),t(2,1_3,\bar1),t(\bar1,2,1_2,\bar1),t(\bar1,1,2,1,\bar1),t(\bar1,1_2,2,\bar1),t(\bar1,1_3,\bar2),
t(\bar1,\{1\}_5)\end{array}\right\};\\
&\textbf{SB}_1=\left\{S(\bar1)\right\},\\
&\textbf{SB}_2=\left\{S(2),S(\bar1,1),S(\bar1,\bar1)\right\},\\
&\textbf{SB}_3=\left\{S(3),S(2,1),S(2,\bar1),S(\bar2,1),S(\bar2,\bar1),S(\bar1,\bar2)\right\},\\
&\textbf{SB}_4=\left\{\begin{array}{l}
S(4),S(3,1),S(3,\bar1),S(\bar3,1),S(\bar3,\bar1),S(\bar1,3),S(\bar1,\bar3),S(2,2),\\
S(2,\bar1,1),S(2,\bar1_2),S(\bar2,1_2),S(\{\bar1,1\}_2)\end{array}\right\},\\
&\textbf{SB}_5=\left\{\begin{array}{l} S(5),S(4,1),S(4,\bar1),S(\bar4,1),S(\bar4,\bar1),S(\bar1,4),S(\bar1,\bar4),S(3,2),S(3,\bar2),S(\bar3,2),\\  S(3,1_2),S(3,1,\bar1),S(3,\bar1,1),S(3,\bar1_2),S(\bar3,\bar1,1),S(\bar1,3,1),S(\bar1,3,\bar1),S(\bar1,\bar3,\bar1),\\
S(2,\bar1,1_2),S(2,\bar1,1,\bar1),S(\bar1,2,1_2),S(\bar1,2,\bar1,1)\end{array}\right\},\\
&\textbf{SB}_6=\left\{\begin{array}{l}
S(6),S(2,4),S(\bar1,5),S(\bar1,\bar5),S(\bar4,1_2),S(\bar4,\bar1,1),S(\bar4,\bar1_2),S(\bar1_2,4),S(3,2,\bar1),S(3,\bar2,1),\\
S(3,1,\bar1,1),S(3,1,\bar1_2),S(3,\bar1,1_2),S(3,\bar1,1,\bar1),S(3,\bar1_2,1),S(3,\bar1_3), S(\bar3,1,\bar1,1), \\
S(\bar1,\bar3,\bar1,1),S(\bar1,\bar3,\bar1_2),S(\bar1,1,3,\bar1),S(2,2,\bar1,1),S(\bar2,2,\bar1,1),S(5,1;\ga,\gb),\\
S(4,2;\ga,\gb),S(1,4,1;-1,\ga,\gb),S(4,1_2;1,\ga,\gb),S(1,3,1,1;-1,1,\ga,\gb):\ga,\gb\in\{\pm1\}
\end{array}\right\};\\
&\textbf{TB}_1=\left\{T(\bar1)\right\},\\
&\textbf{TB}_2=\left\{T(2),T(\bar2)\right\},\\
&\textbf{TB}_3=\left\{T(3),T(\bar3),T(2,\bar1),T(\bar2,1)\right\},\\
&\textbf{TB}_4=\left\{T(4),T(\bar4),T(3,1),T(3,\bar1),T(\bar3,\bar1),T(\bar1,3),T(2,\bar1,1)\right\},\\
&\textbf{TB}_5=\left\{\begin{array}{l} T(5),T(\bar5),T(4,1),T(4,\bar1),T(\bar4,1),T(\bar4,\bar1),T(\bar1,4),T(3,\bar2),\\  T(3,1,\bar1),T(3,\bar1,1),T(3,\bar1,\bar1),T(\bar1,3,1),T(\bar1,3,\bar1)\end{array}\right\},\\
&\textbf{TB}_6=\left\{\begin{array}{l} T(6),T(\bar6),T(5,1),T(5,\bar1),T(\bar5,1),T(\bar5,\bar1),T(\bar1,5),T(\bar1,\bar5),T(\bar4,2),T(2,\bar4),\\ T(4,1_2),T(4,1,\bar1),T(4,\bar1,1),T(4,\bar1_2),T(\bar4,\bar1_2),T(\bar1,4,1),T(\bar1,4,\bar1),T(\bar1,\bar4,\bar1),T(\bar1,1,\bar4),\\
T(3,2,\bar1),T(3,\bar1,1_2),T(3,\bar1,1,\bar1),T(\bar1,3,\bar1_2),T(\bar1_2,3,\bar1)\end{array}\right\}.
\end{align*}

\begin{re} In \cite{XuZhao2020b}, we also studied the $\Q$-vector space generated by the AMTVs of any fixed weight $w$ 
and conjectured that their dimensions $\{d_w\}_{w\ge 1}$ form the tribonacci sequence 1, 2, 4, 7, 13, ....
\end{re}

In the following three examples, we provide some evaluations of a few AMtVs, AMSVs and AMTVs
using the basis given in the above, respectively.

\begin{exa} We have
\begin{align*}
&\begin{aligned}
\tilde{t}(\bar2,\bar1)=-\tfrac1{2}\tilde{t}(3)+\tfrac1{2}\tilde{t}(2,1)=\tfrac1{4}\pi^2\log2-\tfrac7{4}\zeta(3),
\end{aligned}\\
&\begin{aligned}\tilde{t}(\bar1,1,\bar1)=-\tilde{t}(\bar1,\bar2)-\tfrac1{2}\tilde{t}(2,1)=-\pi G+\tfrac{21}{8}\zeta(3),\end{aligned}\\
&\begin{aligned}\tilde{t}(\bar1,\bar1,\bar1)=-\tfrac1{2}\tilde{t}(3)-\tilde{t}(\bar1,\bar2)-\tfrac3{4}\tilde{t}(2,1)
    =-\tfrac1{8}\pi^2\log2-\pi G+\tfrac{35}{16}\zeta(3),\end{aligned}\\
&\begin{aligned}\tilde{t}(\bar2,\bar1,1)={}&-\tfrac5{8}\tilde{t}(4)-\tilde{t}(\bar2,\bar2)-\tfrac1{4}\tilde{t}(3,1)+\tfrac3{4}\tilde{t}(2,1,1)\\
={}&-\tfrac{23}{1440}\pi^4+\tfrac5{24}\pi^2\log^2 2+\tfrac1{6}\log^4 2-2G^2+4\Li_4\big(\tfrac12\big),\end{aligned}\\
&\begin{aligned}
\tilde{t}(\bar1,\bar2,\bar1)={}&2\tilde{t}(\bar4)-2\tilde{t}(2,\bar2)+2\tilde{t}(\bar3,1)-4\tilde{t}(3,\bar1)-\tilde{t}(2,1,\bar1)\\
={}&\tfrac3{16}\pi^3\log2+\tfrac1{12}\pi\log^3 2+\tfrac7{8}\pi\zeta(3)-16\beta(4)+16\Im\Li_4\big(\tfrac{1+i}{2}\big),
\end{aligned}\\
&\begin{aligned}
\tilde{t}(\bar2,\{\bar1\}_3)={}&\tfrac1{40}\tilde{t}(5)+\tfrac9{5}\tilde{t}(\bar4,\bar1)
    -\tfrac{41}{40}\tilde{t}(4,1)-\tfrac1{8}\tilde{t}(3,1,1)+\tfrac1{8}\tilde{t}(2,1_3)\\
={}&-\tfrac{91}{5760}\pi^4\log2+\tfrac1{72}\pi^2\log^3 2+\tfrac1{60}\log^5 2+\tfrac{29}{128}\pi^2\zeta(3)-2\Li_5\big(\tfrac12\big)-\tfrac{155}{256}\zeta(5).
\end{aligned} 
% &\begin{aligned}
% \tilde{t}(\bar1,\bar1,1_3)={}&-\tfrac1{5}\tilde{t}(5)-\tfrac{13}{20}\tilde{t}(\bar4,\bar1)
% +\tfrac{51}{80}\tilde{t}(4,1)+\tilde{t}(2,\bar1,\bar2)+\tfrac1{8}\tilde{t}(3,1,1)-\tilde{t}(\bar1,2,1,\bar1)-\tfrac7{16}\tilde{t}(2,1_3)\\
% ={}&\pi\beta(4)-4\pi\Im\Li_4\big(\tfrac{1+i}{2}\big)-\tfrac{19}{576}\pi^4\log2-\tfrac1{144}\pi^2\log^3 2-\tfrac1{48}\log^5 2\\
% &\quad-\tfrac1{4}\pi^2\zeta(3)+\tfrac5{2}\Li_5\big(\tfrac12\big)+\tfrac{403}{64}\zeta(5).
% \end{aligned}
\end{align*}
\end{exa}

\begin{exa} We have
\begin{align*}
&\begin{aligned}
S(\bar1,2)=-\tfrac2{3}S(\bar2,\bar1)+\tfrac1{3}S(\bar2,1)=\tfrac{11}{96}\pi^3-4\Im\Li_3\big(\tfrac{1+i}{2}\big)-2G\log2+\tfrac1{8}\pi\log^2 2,
\end{aligned}\\
&\begin{aligned}
S(\bar1,\bar1,2)={}&\tfrac8{3}S(\bar1,\bar3)-\tfrac{19}{15}S(\bar1,3)-\tfrac{8}{5}S(\bar3,\bar1)+\tfrac4{5}S(\bar3,1)-2S(2,\bar1,\bar1)+S(2,\bar1,1)\\
={}&-6\beta(4)+8\Im\Li_4\big(\tfrac{1+i}{2}\big)+4\Im\Li_3\big(\tfrac{1+i}{2}\big)\log2+G\log^2 2-\tfrac1{12}\pi\log^3 2,
\end{aligned}\\
&\begin{aligned}
S(\bar1,1,\bar1,\bar1)={}&\tfrac{15}{16}S(4)+\tfrac1{8}S(2,2)-\tfrac1{4}S(\bar2,1,1)+\tfrac3{4}S(\bar1,1,\bar1,1)\\
={}&2G^2+\tfrac9{64}\pi^4-6\pi\Im\Li_3\big(\tfrac{1+i}{2}\big)-2\pi G\log2+\tfrac3{16}\pi^2\log^2 2-\tfrac7{4}\zeta(3)\log2,
\end{aligned}\\
&\begin{aligned}
S(\bar3,1,\bar1)={}&-\tfrac{31}{112}S(5)-\tfrac{13}{112}S(3,2)-\tfrac{3}{16}S(4,1)+\tfrac1{4}S(3,1,1),
\end{aligned}\\
&\begin{aligned}
S(\bar3,1,2)={}&\tfrac{35}{128}S(6)-\tfrac{25}{192}S(2,4)+\tfrac{25}{192}S(4,2)-\tfrac1{3}S(4,1,\bar1)+\tfrac1{6}S(4,1,1)\\
={}&\tfrac{247}{967680}\pi^6+\tfrac1{288}\pi^4\log^2 2-\tfrac1{288}\pi^2\log^4 2
-\tfrac7{96}\pi^2\zeta(3)\log2+\tfrac7{16}\ze^2(3)-\tfrac1{12}\pi^2\Li_4\big(\tfrac12\big).
\end{aligned}
\end{align*}
\end{exa}

\begin{exa} We have
\begin{align*}
&\begin{aligned}
T(\bar2,\bar1)=-2T(\bar3)+2T(2,\bar1)=-\tfrac3{16}\pi^3-\tfrac1{4}\pi\log^2 2+4G\log2+8\Im\Li_3\big(\tfrac{1+i}{2}\big),
\end{aligned}\\
&\begin{aligned}
T(\bar1,1,1,\bar1)=\tfrac1{2}T(\bar4)+\tfrac3{2}T(\bar3,\bar1)=-\tfrac1{4}\pi^2G+2\beta(4),
\end{aligned}\\
&\begin{aligned}
T(\bar1,1,\bar1,\bar1)=2T(\bar3,\bar1)+T(3,\bar1)=-\tfrac7{8}\pi\zeta(3)-\tfrac1{4}\pi^2G+6\beta(4),
\end{aligned}\\
&\begin{aligned}
T(\bar3,\bar1,1)={}&\tfrac6{5}T(\bar5)+2T(3,\bar2)+6T(\bar4,\bar1)-6T(4,\bar1)+2T(3,\bar1,\bar1)\\
={}&-\tfrac7{128}\pi^5-\tfrac1{16}\pi^3\log^2 2+\pi^2G\log2+2\pi^2\Im\Li_3\big(\tfrac{1+i}{2}\big),
\end{aligned}\\
&\begin{aligned}
T(\bar3,1_3)={}&-\tfrac5{48}T(6)-T(\bar5,1)+T(5,1)-\tfrac1{4}T(4,1,1)\\
={}&-\tfrac{509}{322560}\pi^6-\tfrac{\pi^4}{96}\log^2 2+\tfrac{\pi^2}{96}\log^4 2
    +\tfrac{7\pi^2}{32}\zeta(3)\log2+\tfrac{\pi^2}{4}\Li_4\big(\tfrac12\big)-\tfrac{97}{128}\ze^2(3)+\zeta(\bar5,1),
\end{aligned}\\
&\begin{aligned}
T(\bar1,2,1_3)=-\tfrac1{2}T(\bar1,\bar5)+\tfrac{13}{2}T(\bar5,\bar1)+T(5,\bar1).
\end{aligned}
\end{align*}
\end{exa}

Supported by theoretical and numerical evidence presented in this paper, we now formulate a few conjectures concerning the dimensions of some subspaces of AMMVs. These conjectures hint at a few very rich  but previously overlooked algebraic and geometric structures associated with these vector spaces. It is our hope that these unsolved problems will provoke further development and more detailed studies of the motivic theory of the AMMVs.

\begin{con} \emph{(cf.\ \cite[Conj. 5.2]{XuZhao2020b})} \label{conj:tribonacci2}
We have the following generating function
\begin{equation*}
    \sum_{n=0}^\infty (\dim_\Q \AMTV_n)t^n = \frac{1}{1-t-t^2-t^3}.
\end{equation*}
Namely, the dimensions form the tribonacci sequence $\{d_w\}_{w\ge 1}=\{1,2,4,7,13,24,\dotsc\}$, see A000073 at oeis.org.
\end{con}

\begin{con}  \label{conj:AMSV} We have $\dim_\Q \AMSV_1=1$, $\dim_\Q \AMSV_2=3$, and
\begin{equation*}
     \dim_\Q \AMSV_n=2\dim_\Q \AMSV_{n-1}-2\Big\lfloor \frac{n+1}{2} \Big\rfloor +4 \quad \text{for all $n\geq 3$}.
\end{equation*}
\end{con}

\begin{con}  \label{conj:AMtV} For $n\geq 2$, we have
\begin{equation*}
     \dim_\Q \AMtV_n=3 \times 2^{n-2}.
\end{equation*}
\end{con}

\begin{con}\label{conj:AMMVoinAMMVe}
For all weight $w\ge 1$ we have
\begin{enumerate}
  \item [\upshape{(i)}] $\AMMVo_w\subsetneq \AMMVe_w=\AMMV_w$,
  \item [\upshape{(ii)}] $\dim_\Q \AMMVo_w=2^w-1$, $\dim_\Q \AMMV_w=2^w$, and
  \item [\upshape{(iii)}] $\AMMV_w=\AMMVo_w\oplus  \Q\log^w 2$.
\end{enumerate}
\end{con}

Finally, from numerical evidence, we have the following conjectural diagram showing
relations between all the different subspaces of $\AMMV$.
\begin{center}
\includegraphics[height=2.2 in]{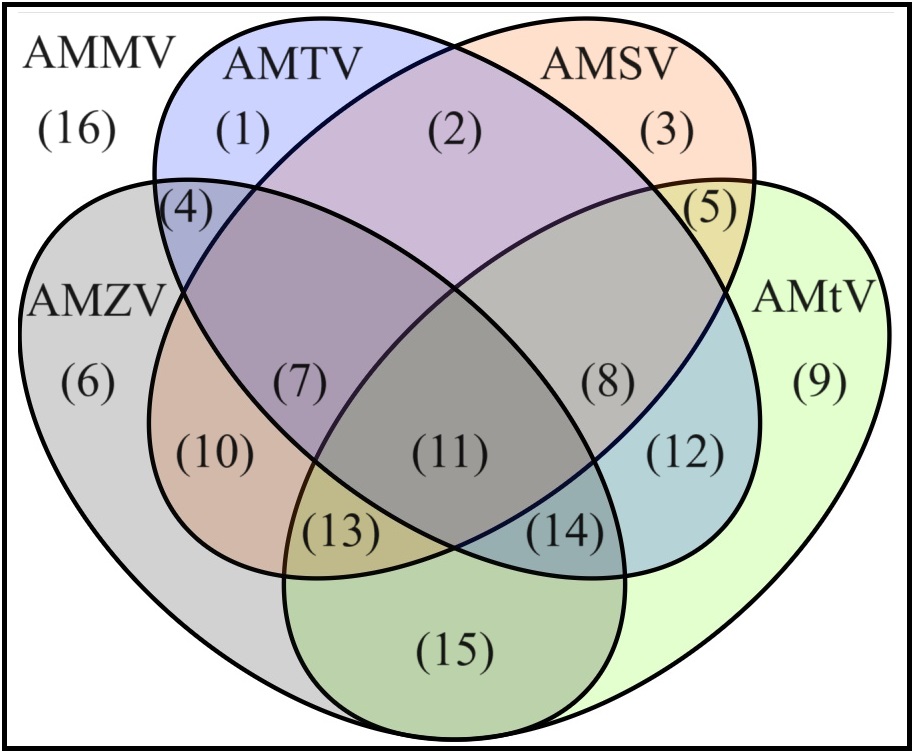}
\end{center}

In the following table, we provide at least one element (if we found any)
in each of the above 16 regions, which has been verified
numerically using the command \texttt{lindep} in GP-Pari with up to 4000 digits of precision.
\begin{table}[!h]
{
\caption{Inclusion relations between various subspaces of $\AMMV$.}
\begin{center}
\begin{tabular}{  |c|c|c|c|c|c|c|c|c|} \hline
 Region        & (1) & (2)& (3)& (4)& (5)& (6)& (7)& (8)\\ \hline
Element  &$\emptyset$ &  $T(2,\bar1)$  &  $S(2,\bar1,1)$  &$\emptyset$ &  $S(\bar3,1)$  & $\log^4 2$  & $\zeta(\bar3,1)$ & $T(\bar4)$\\ \hline\hline
 Region        & (9) & (10) & (11) & (12) & (13) & (14) & (15) & (16)   \\ \hline
Element  &$t(\bar1,1,1)$ &  $S(3,1)$  & $\zeta(4)$  & $T(\bar2,1)$  & $S(4,1)$ &$\emptyset$&  $\zeta(2)\log^2 2$ & $M(\cbar1, 1_4)$\\ \hline
\end{tabular}
\end{center}
}
\label{Table:AMMVregion}
\end{table}

After extensive search up to weight 8, we found that only one AMTV, i.e., $T(\bar2,1)$,
lies outside of $\AMSV$. Hence, we propose the following conjecture.
\begin{con}\label{conj:AMTVinAMSV}
For all $w\ge 4$ we have the inclusion $\AMTV_w\subsetneq \AMSV_w$.
\end{con}

Turing to the non-alternating MMVs, let $\MMV$ be the $\Q$-span of all MMVs. We know from \cite[Theorem 7.1]{XuZhao2020a} that 
the space $\MMV$ should be included properly in the space $\AMZV$ of alternating MZVs. By numerical computation up to weight 12,
we have found the following diagram showing the relations between different subspaces of $\MMV$.

\begin{center}
\includegraphics[height=1.8 in]{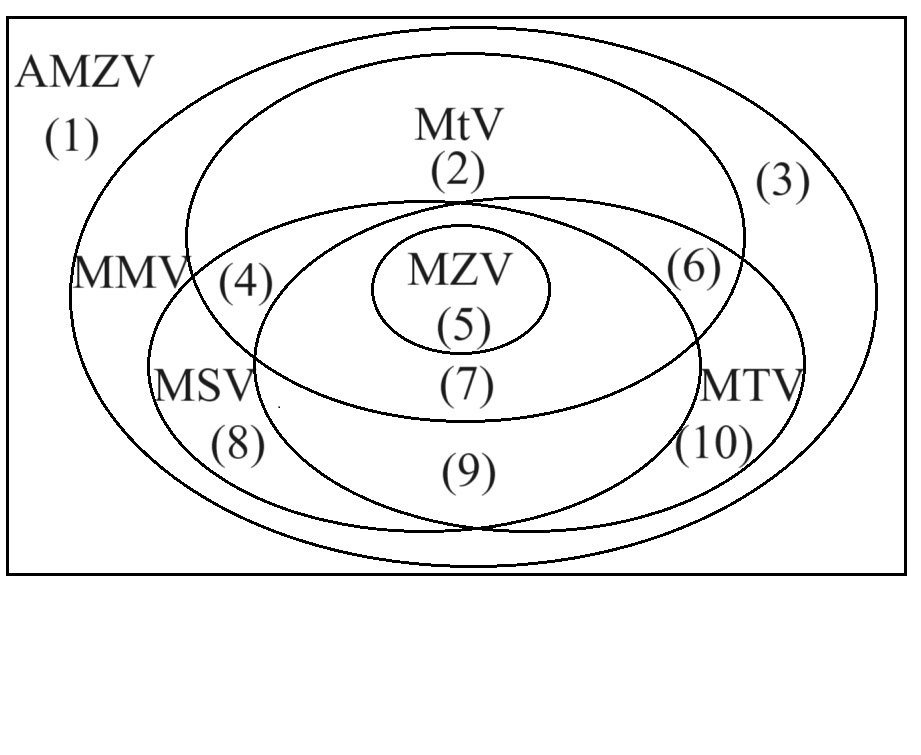}
\end{center}

In the following table, similarly to the AMMV case, we provide at least one element (if we found any)
in each of the above 10 regions, which has been verified numerically by GP-Pari. 

\begin{table}[!h]
{
\caption{Inclusion relations between various subspaces of $\MMV$.}
\begin{center}
\begin{tabular}{  |c|c|c|c|c|c|c|c|c|c|c|} \hline
 Region      & (1) & (2)& (3)& (4)& (5)& (6)& (7)& (8) & (9) & (10)\\ \hline
Element   &  $\log 2$  &  $t(2,1)$  & $M(\wc2,1_4)$  & $S(4,1_2)$  & $\zeta(2) $ &  $\emptyset$  &  $S(5,1,2)$ & $S(2,1)$  & $S(2,2)$  & $\emptyset$  \\ \hline\hline
\end{tabular}
\end{center}
}
\label{Table:MMVregion}
\end{table}

After extensive search up to weight 12, we found that no MTV lies outside the space of MSVs.
Note that $\MTV_2=\MSV_2=\MZV_2=\pi^2\Q$. Hence, we propose the following conjecture.
\begin{con}\label{conj:MTVinMSV}
For all weight $w\ge 3$ we have the inclusion $\MTV_w\subsetneq \MSV_w$.
\end{con}

Conjecture~\ref{conj:AMTVinAMSV} and Conjecture~\ref{conj:MTVinMSV} are consistent in spirit with
Conjecture~\ref{conj:AMMVoinAMMVe} and the following conjecture (see \cite[Table 1]{XuZhao2020a}).

\begin{con}\label{conj:MMVoinMMVe}
Let $\MMVe_w$ (resp.\ $\MMVo_w$) be the $\Q$-span of all MMVs with the smallest index being even
(resp.\ odd). Then 
\begin{enumerate}
  \item [\upshape{(i)}] $\MMVe_w=\MMV_w$ for all $w\ge 1$, and
  \item [\upshape{(ii)}] $\MMVo_w\subsetneq \MMV_w$ for all $w\ge 5$.
\end{enumerate} 
\end{con}

\bigskip
\noindent
{\bf Acknowledgments.} Ce Xu expresses his deep gratitude to Dr. Li Lai and Kamcheong Au for valuable discussions and comments. All the authors are very grateful to the referee whose extremely detailed comments, suggestions and questions have prompted them to dig much deeper into the subtle structures of relevant objects of this paper. In particular, his suggestion to use generating functions to prove the general parity result is very enlightening. Ce Xu is supported by the National Natural Science Foundation of China (Grant No. 12101008), the Natural Science Foundation of Anhui Province (Grant No. 2108085QA01) and the University Natural Science Research Project of Anhui Province (Grant No. KJ2020A0057). Lu Yan is supported by the Fundamental Research Funds for the Central Universities. Jianqiang Zhao is supported by the Jacobs Prize from The Bishop's School.

\medskip
\noindent
\textbf{Declaration of Competing Interest.} The authors declare that there is no competing interest.

\end{document}